\documentclass{amsart}

\usepackage{amssymb}
\usepackage{amsmath}
\usepackage{amsthm}

\newtheorem{theorem}{Theorem}[section]
\newtheorem{corollary}[theorem]{Corollary}
\newtheorem{lemma}[theorem]{Lemma}
\newtheorem{example}[theorem]{Example}
\newtheorem{claim}[theorem]{Claim}

\newtheorem*{lemma1*}{Theorem \ref{k-welldef}}
\newtheorem*{lemma2*}{Theorem \ref{TWgen}}
\newtheorem*{lemma3*}{Theorem \ref{bb2}}
\theoremstyle{definition}
\newtheorem{definition}[theorem]{Definition}
\newtheorem{question}[theorem]{Question}
\newtheorem*{def1*}{Definition \ref{k-Solovay}}

\usepackage{ stmaryrd }
\usepackage{ bbold }

\usepackage{tikz-cd}

\DeclareMathOperator{\cof}{cof}

\DeclareMathOperator{\Ob}{Ob}
\DeclareMathOperator{\crit}{crit}
\DeclareMathOperator{\mc}{mc}
 \DeclareMathOperator{\Coll}{Coll}
 \DeclareMathOperator{\dom}{dom}
\DeclareMathOperator{\id}{id}
\renewcommand{\le}{\leqslant}
\renewcommand{\ge}{\geqslant}
\renewcommand{\leq}{\leqslant}
\renewcommand{\geq}{\geqslant}
\newcommand\hacek\check
\newcommand\cat[1]{{}^\curvearrowright #1}

\title{Higher Solovay Models}

\author[Straffelini]{Cesare Straffelini}
 \address
       {Departament de Matemàtiques i Informatica, Universitat de Barcelona, Gran Via de les Corts Catalanes, 585, 08007 Barcelona, Catalunya; \& Dipartimento di Matematica, Università degli Studi di Trento, Via Sommarive, 14, 38123 Povo (Trento), Italia}
\email{straffelini@ub.edu \& cesare.straffelini@unitn.it}

\author[Thei]{Sebastiano Thei}
\address{Dipartimento di Scienze Matematiche, Informatiche e Fisiche, Università degli Studi di Udine, Via delle Scienze, 206, 33100 Udine, Italia}
\email{thei91.seba@gmail.com}

\begin{document}


\date{\today}


\begin{abstract}
We introduce an axiomatisation of when a model of the form $L(V_{\kappa+1})^M$ can be considered a ``$\kappa$-Solovay model''; we show a characterisation of $\kappa$-Solovay models; and we prove elementary equivalences between $\kappa$-Solovay models.
\end{abstract}

\maketitle

\section{Introduction}


Robert Solovay, in his renowned article \cite{Sol70}, was able to construct a model of set theory (\textsf{ZF}) in which every subset of the real line is very ``regular'', for example it is Lebesgue measurable, has the Baire and perfect set properties, and many others. He built this model starting with a model $V$ of \textsf{ZFC} with an inaccessible cardinal $\lambda$, then considering the forcing extension $V[G]$ of $V$ resulting from collapsing gently $\lambda$ to $\omega_1$ with the Lévy collapse $\Coll(\omega, {<}\lambda)$, and after that taking the inner model $L(\mathbb R)$ of this extension $V[G]$. This model $L(\mathbb R)^{V[G]}$ is what for years has been called informally ``the'' Solovay model, even though it is not unique, depending from $\lambda$.\medskip

Years later, Joan Bagaria and Hugh Woodin in \cite{BW} proved a simple abstract characterisation of when a model of the form $L(\mathbb R)^M$ is ``a'' Solovay model:

\begin{definition}\label{Solovay}
Let $V\subseteq M$ be models of \textsf{ZFC}. We will say that the model $L(\mathbb R)^M$ is \emph{a Solovay model} over $V$ if $M$ satisfies the following conditions:
\begin{itemize}
\item for all $x\subseteq\omega$, the ordinal $\omega_1^M$ is an inaccessible cardinal in $V[x]$;
\item for all $x\subseteq\omega$, $x$ is \emph{small-generic} over $V$, that is, there is a countable poset $\mathbb P$ in $V$ and a filter $g\subseteq\mathbb P$ generic over $V$ such that $V[x]\subseteq V[g]$.
\end{itemize}
\end{definition}

Clearly, the original model by Solovay satisfies this definition.

\begin{theorem}[Solovay]\label{welldef}
If $\lambda$ is inaccessible in $V$, and $G\subseteq\Coll(\omega,{<}\lambda)$ is a filter generic over $V$, then the model $L(\mathbb R)^{V[G]}$ is a Solovay model over $V$.
\end{theorem}

\begin{proof}
This theorem was proven in the original article \cite[Section 3.4]{Sol70}.
\end{proof}

Bagaria and Woodin showed in \cite{BW} that the converse is essentially true:

\begin{theorem}[Bagaria-Woodin]\label{BagariaWoodin}
If $V\subseteq M$ are models of \textsf{ZFC} with $L(\mathbb R)^M$ Solovay over $V$, there is a filter $C\subseteq\Coll(\omega, {<}\omega_1^M)$ generic over $V$ with $L(\mathbb R)^M=L(\mathbb R)^{V[C]}$.
\end{theorem}

\begin{proof} The proof of this theorem can be found in \cite[Theorem 2.11]{BW}, or written in a way that is more similar to our presentation here in \cite[Lemma 1.2]{BB04b}.
\end{proof}

In addition to this, Solovay models have many absoluteness properties, as proved by Joan Bagaria and Roger Bosch in their articles \cite{BB04} and \cite{BB04b}. For example,

\begin{theorem}[Bagaria-Bosch]\label{bb}
    Suppose that $L(\mathbb R)^M$ and $L(\mathbb R)^N$ are Solovay models over the same model $V$, such that $\mathbb R^M\subseteq \mathbb R^N$ and $\omega_1^M=\omega_1^N$. Then there exists an unique elementary embedding $j:L(\mathbb R)^M\hookrightarrow L(\mathbb R)^N$ that fixes all the ordinals.
\end{theorem}

\begin{proof} The proof of this theorem can be found in \cite[Lemma 1.5]{BB04b}.
\end{proof}

We develop an abstract framework to show that the axiomatic characterization of $L(\mathbb{R})$, started in 1970 by Solovay (Theorem \ref{welldef}) and continued in the following decades by Bagaria, Bosch and Woodin (Theorems \ref{BagariaWoodin} and \ref{bb}) can be shifted upwards to the uncountable level. 

\begin{def1*}
Let $V\subseteq M$ be models of \textsf{ZFC}. We say that the model $L(V_{\kappa+1})^M$ is a \emph{$\kappa$-Solovay model} over $V$ if $M$ satisfies the following conditions:
\begin{itemize}
\item for all subsets $x\subseteq\kappa$, the ordinal $(\kappa^+)^M$ is an inaccessible cardinal in $V[x]$;
\item for $x\subseteq \kappa$ in $M$, there is a $\Sigma$-Příkrý forcing $\mathbb P\in V$ with $|\mathbb P|^M\le\kappa$, and $g\subseteq\mathbb P$ generic over $V$ with $V[x]\subseteq V[g]$ and $\cof(\kappa)^{V[g]}=\omega$.
\end{itemize}
\end{def1*}

The three main results of this paper are higher analogues of Theorems \ref{welldef}, \ref{BagariaWoodin} and \ref{bb}:

\begin{lemma1*}
    Let $\mathbb P_\infty$ be a $(\kappa,\lambda)$-nice weak $\Sigma$-system, and let $G\subseteq \mathbb P_\infty$ be a generic filter over $V$. Then $L(V_{\kappa+1})^{V[G]}$ is $\kappa$-Solovay over $V$.
\end{lemma1*}

\begin{lemma2*}
Suppose $V\subseteq M$ are models of \textsf{ZFC} and assume that $L(V_{\kappa+1})^M$ is a $\kappa$-Solovay model over $V$. Let $\mathbb P_\infty$ be a $(\kappa,(\kappa^+)^M)$-nice weak $\Sigma$-system in $V$. Then there is a forcing notion $\mathbb W$ in $M$ (the \emph{constellating poset}) 
such that in the forcing extension by $\mathbb W$ there is a filter $C\subseteq\mathbb P_\infty$ generic over $V$ such that $M$ and $V[C]$ have the same subsets of $\kappa$, hence the same $L(V_{\kappa+1})$.
\end{lemma2*}

\begin{lemma3*}
Suppose $V\subseteq M$ and $V\subseteq N$ are models of \textsf{ZFC} such that $L(V_{\kappa+1})^M$ and $L(V_{\kappa+1})^N$ are $\kappa$-Solovay models over $V$, $\wp(\kappa)^M\subseteq \wp(\kappa)^N$ and $(\kappa^+)^M=(\kappa^+)^N$. Suppose there is a $(\kappa,(\kappa^+)^M)$-nice weak $\Sigma$-system in $V$. Then there is a unique
\[j:L(V_{\kappa+1})^M\hookrightarrow L(V_{\kappa+1})^N\]
that is an elementary embedding and fixes all the ordinals.
\end{lemma3*}

Let us briefly motivate the above theorems. We live in a time where many of the classical results about $L(\mathbb R)$ get generalised for $L(V_{\kappa+1})$ for $\kappa$ a cardinal (usually, with some additional requirements), so it has not been a surprise when the second author of this article, together with Vincenzo Dimonte and Alejandro Poveda, published the article \cite{DPT}, in which they built a model of $\textsf{ZF}+\textsf{DC}_{\kappa}$ in which every subset of $\kappa^\omega$ has the $\kappa$-perfect set property, and every subset of $\prod_{n<\omega} \kappa_n$ has the ${{\mathcal U}}$-Baire property. These are higher analogues to this new context of the statements ``every set of reals has the perfect set property'' and ``every set of reals has the Baire property'', that hold true in the Solovay model.\medskip

Curiously, the first proof of the consistency of the $\kappa$-perfect set property relied on the structural properties of the model $L(V_{\kappa+1})$, examined through the lens of the axiom $I_0(\kappa)$, one of the strongest known large cardinal axioms. Specifically, work by Woodin in \cite
{WoodinPartII}, and by his students Cramer in \cite
{Cramer} and Shi in \cite
{Shi}, has revealed that under $I_0(\kappa)$ every set $A\subseteq \kappa^\omega$ in $L(V_{\kappa+1})$ has the $\kappa$-perfect set property. Along the same lines, Dimonte, Iannella and L{\"u}cke in \cite{DimonteIannella} showed that, under $I_0(\kappa)$, all projective subsets of $\prod_{n<\omega} \kappa_n$ have the ${\mathcal{U}}$-Baire Property (for a certain ${\mathcal{U}}$).\medskip

These results are pieces of evidence suggesting that the investigation of regularity properties of sets in $ L(V_{\kappa+1})$ under $I_0(\kappa)$ can be regarded as a natural generalisation of the classical study of regularity properties of sets in $L(\mathbb{R})$ under $\textsf{AD}^{L(\mathbb{R})}$. Indeed, Woodin realised that $L(V_{\kappa+1})$ behaves under $I_0(\kappa)$ very much like $L(\mathbb{R})$ does under $\textsf{AD}^{L(\mathbb{R})}$  (see, for example, \cite{Koel} or \cite{Cramer}). For instance, $L(V_{\kappa+1})\models ``\kappa^+\text{ is measurable''}$ under $I_0(\kappa)$, while $L(\mathbb{R})\models``\omega_1\text{ is measurable''}$ under $\textsf{AD}^{L(\mathbb{R})}$.\medskip

In light of this, axiom $I_0(\kappa)$ provides arguably the appropriate axiomatic framework to develop the so-called generalised descriptive set theory in $\kappa$-Polish spaces at singular cardinals $\kappa$. A notable key point for this parallelism is that such a cardinal $\kappa$ shares with $\omega$ some crucial properties which are used to obtain the various results in classical descriptive set theory. For instance, like $\omega$, $\kappa$ must be a strong limit cardinal of countable cofinality.\medskip


Taking cue from some descriptive set-theoretical aspects of $I_0(\kappa)$, Dimonte and Motto Ros developed a generalised descriptive set theory in their book \cite{DimonteMottoShi} on strong limit uncountable cardinals of countable cofinality under $I_0(\kappa)$. However, there is evidence suggesting that $I_0(\kappa)$ may not be the optimal large cardinal assumption for ensuring regularity properties for subsets in $\kappa$-Polish spaces. Indeed, as pointed out in the first paragraph, Solovay in \cite{Sol70} used a large cardinal assumption which is much lower than $\textsf{AD}^{L(\mathbb R)}$. Similarly, the authors in \cite{DPT} obtained the same configuration in the generalised context, relying on large cardinal assumptions in the realm of supercompactness.\medskip 

To keep on with the similarities, the model introduced in \cite{DPT} is the $L(V_{\kappa+1})$ of a forcing extension of $V$. Here, though, is where the similarities between this new work and Solovay's historic \cite{Sol70} stop, since the forcing notion used by Dimonte, Poveda and the second author is very different than the Lévy collapse.\medskip

This paper aims to explore the interaction between generalised descriptive set theory and large cardinal axioms through the lens of higher Solovay models. These models provide a natural setting for studying regularity properties in $\kappa$-Polish spaces that hold true in $L(V_{\kappa+1})$ under $I_0(\kappa)$.\medskip

This article is structured as follows. In Section 2, we introduce all preliminaries about forcing that will be needed throughout the paper. Many of these are already stated in \cite{DPT}, but we report them for sake of completeness. Later, in Section 3, we recall the concept of $\Sigma$-Příkrý forcings and $(\kappa,\lambda)$-nice weak $\Sigma$-systems, the technical instruments used in \cite{DPT} for building their higher Solovay model. We strengthen the notion of these weak $\Sigma$-systems by adding an hypothesis of continuity, that is necessary later.\medskip

In Section 4, we recall the two main examples of $(\kappa,\lambda)$-nice weak $\Sigma$-systems from \cite{DPT} and check that they satisfy the additional hypothesis of continuity. In Section 5, we finally generalise Definition \ref{Solovay}, explaining when a model of the form $L(V_{\kappa+1})^M$ can be called a \emph{$\kappa$-Solovay model} over $V$. Concretely, we get immediately the analogous of Theorem \ref{welldef}, namely, the models built by Dimonte, Poveda and the second author in \cite{DPT} are indeed satisfying our definition of $\kappa$-Solovay model.\medskip


Section 6 will be dedicated to proving the higher analogue of Bagaria-Woodin's Theorem \ref{BagariaWoodin}. In Section 7, we see how this theorem has many consequences in the generalised descriptive set theory, showing that every $\kappa$-Solovay model satisfies the desired regularity properties for subsets of $\kappa$-Polish spaces.\medskip


To wrap up, in Section 8 we will prove a version of Bagaria-Bosch's Theorem \ref{bb} adapted to $\kappa$-Solovay models. Some open questions are discussed in Section 9.

\section{Forcing Preliminaries}

In this section, we collect some definitions and basic results about forcing that are necessary for understanding the rest of the article. Our main reference is Kenneth Kunen's book \cite{Kunen}. Formally, a forcing poset is an ordered triple $\langle\mathbb P,\le_\mathbb P,\mathbb 1_\mathbb P\rangle$ where $\le_\mathbb P$ is a partial order relation on $\mathbb P$ with a maximum element, that is $\mathbb 1_\mathbb P$. However, when there is no risk of confusion, we often adopt a more informal style, as in ``\emph{let $\mathbb P$ be a forcing poset}''.\medskip

An important definition we use, which does not appear on Kunen's book, is that of a \emph{weak projection}, introduced by Foreman and Woodin in their article \cite{FW}:

\begin{definition}\label{weakproj}
If $\langle \mathbb P, \le_{\mathbb P}, \mathbb 1_{\mathbb P}\rangle$ and $\langle\mathbb Q, \le_{\mathbb Q}, \mathbb 1_{\mathbb Q}\rangle$ are forcing posets, a map $\pi:\mathbb P\to\mathbb Q$ is called \emph{weak projection} if the following conditions hold:
\begin{itemize}
    \item for all $p$ and $p'$ in $\mathbb P$, if $p\le_{\mathbb P} p'$, then $\pi(p)\le_{\mathbb Q} \pi(p')$;
    \item for all $p\in\mathbb P$, there exists $p^\ast\le_{\mathbb P} p$ such that for all $q\in\mathbb Q$ with $q\le_{\mathbb Q} \pi(p^\ast)$ there exists $p'\le_{\mathbb P} p$ satisfying $\pi(p')\le_{\mathbb Q} q$.
\end{itemize}
A \emph{projection} is a weak projection $\pi$ that also satisfies $\pi(\mathbb 1_\mathbb P)=\mathbb 1_\mathbb Q$ and for which, in the second condition above, one can always take $p^\ast=p$.
\end{definition}

\begin{lemma}\label{filterweak}
    If $\pi:\mathbb P\to\mathbb Q$ is a weak projection between forcing posets, and $G$ is a $\mathbb P$-generic filter over $V$, then the upward closure of $\pi''G$ is a $\mathbb Q$-generic filter over $V$.
\end{lemma}

\begin{proof}
   The proof of this lemma can be found in \cite[Proposition 2.8]{FW}.
\end{proof}

As a convention, following Lemma \ref{filterweak}, we identify $\pi''G$ with its upwards closure.

\begin{example}\label{lévy} Let us consider the case of the Lévy collapse. Namely, let $\Coll(\omega, {<}\lambda)$ be the set of all partial maps from $\omega$ to some $\alpha<\lambda$, ordered by reverse inclusion. In the forcing extension by $\Coll(\omega, {<}\lambda)$, $\lambda$ is collapsed gently to $\omega_1$, in the sense that every ordinal below $\lambda$ becomes countable. On the other hand, let $\Coll(\omega, {\le}\alpha)$ be the set of all partial functions from $\omega$ to $\alpha$, again with the reverse inclusion ordering. In the extension by $\Coll(\omega, {\le}\alpha)$, the ordinal $\alpha$ becomes countable. Next, if $\alpha<\lambda$, we observe that \[\Coll(\omega,{\le}\alpha)\ast \Coll(\hacek\omega,{<}\hacek\lambda)\simeq\Coll(\omega,{<}\lambda),\] hence there is a (strong) projection $\ell_{\lambda,\alpha}:\Coll(\omega,{<}\lambda)\to\Coll(\omega,{\le}\alpha)$.\end{example}

A concept in some ways dual to projections is the one of \emph{complete embeddings}:

\begin{definition}
    If $\langle \mathbb P, \le_{\mathbb P}, \mathbb 1_{\mathbb P}\rangle$ and $\langle\mathbb Q, \le_{\mathbb Q}, \mathbb 1_{\mathbb Q}\rangle$ are forcing posets, a map $\imath:\mathbb Q\to\mathbb P$ is a \emph{complete embedding} if the following conditions hold:
    \begin{itemize}
        \item for all $q$ and $q'$ in $\mathbb Q$, if $q\le_\mathbb Q q'$, then $\imath(q)\le_\mathbb P \imath(q')$;
        \item for all $q$ and $q'$ in $\mathbb Q$, if $q$ and $q'$ are incompatible in $\mathbb Q$, then $\imath(q)$ and $\imath(q')$ are incompatible in $\mathbb P$;
        \item for every $p\in\mathbb P$, there exists $q\in\mathbb Q$ such that $p$ and $\imath(q)$ are compatible in $\mathbb P$.
    \end{itemize}
\end{definition}

\begin{lemma}
    If $\imath:\mathbb Q\to\mathbb P$ is a complete embedding between forcing posets, and $G$ is a $\mathbb P$-generic filter over $V$, then $\imath^{-1}[G]$ is a $\mathbb Q$-generic filter over $V$.
\end{lemma}

\begin{proof}
    The proof of this classical lemma can be found in \cite[Lemma IV.4.2]{Kunen}.
\end{proof}

Given a forcing poset $\mathbb P$, we can consider its Boolean completion, the \emph{regular open algebra}, where we remove the least element $\mathbb 0_{\mathbb P}$. The bottomless regular open algebra is denoted by $\mathrm{RO}(\mathbb P)$. For sake of simplicity, we identify the forcing poset $\mathbb P$ with its isomorphic copy inside $\mathrm{RO}(\mathbb P)$. The forcing posets $\mathbb P$ and $\mathrm{RO}(\mathbb P)$ are always forcing equivalent, as $\mathbb P$ is dense in $\mathrm{RO}(\mathbb P)$ and every $\mathbb P$-name is also a $\mathrm{RO}(\mathbb P)$-name.\medskip

 If $\mathbb Q$ is another forcing poset and $\varrho:\mathrm{RO}(\mathbb P)\to\mathrm{RO}(\mathbb Q)$ is a projection, then also the restriction of $\varrho$ to $\mathbb P$ is a projection. Similarly, by density, if $\imath:\mathbb P\to\mathrm{RO}(\mathbb Q)$ is a complete embedding, it can be extended in an unique way to an embedding $\imath:\mathrm{RO}(\mathbb P)\to\mathrm{RO}(\mathbb Q)$, by defining
\[\imath(b):=\bigvee\{\imath(p): p\in\mathbb P,\ p\le b\}.\]

\begin{lemma}\label{26}
   If $G$ is $\mathrm{RO}(\mathbb P)$-generic over $V$, then $G\cap\mathbb P$ is $\mathbb P$-generic over $V$. On the other hand,
if $G$ is $\mathbb P$-generic over $V$, its upwards closure in $\mathrm{RO}(\mathbb P)$ is $\mathrm{RO}(\mathbb P)$-generic over $V$. 

\end{lemma}

\begin{proof}
        The proof of this lemma can be found in \cite[Lemma IV.4.7]{Kunen}.
\end{proof}

Forcing with the Boolean algebra $\mathrm{RO}(\mathbb P)$ instead of $\mathbb P$ has the advantage that for every sentence $\varphi$ in the forcing language (namely, every set-theoretic formula having only $\mathbb P$-names as parameters) that is forced by some $p\in\mathbb P$, we can always consider the \emph{weakest} condition in $\mathrm{RO}(\mathbb P)$ that forces it, called its \emph{Boolean value}. Namely, we can define it as
\[\llbracket \varphi\rrbracket:=\bigvee \{p\in\mathbb P: p\Vdash_{\mathbb P} \varphi\}.\]
For more facts about Boolean-valued forcing, we refer to John Bell's book \cite{Bell}.

\begin{definition}
    Let $\mathbb P$ and $\mathbb Q$ be forcings and let $\dot g$ be a $\mathbb P$-name such that \[\mathbb 1_{\mathbb P}\Vdash_{\mathbb P} ``\dot g\ \text{is}\ \hacek{\mathbb Q}\text{-generic}".\]
    Then the function $\imath:\mathbb Q\to\mathrm{RO}(\mathbb P)$ defined by $\imath(q):=\llbracket \hacek q\in \dot g\rrbracket$ for all $q\in\mathbb Q$ is called the \emph{complete embedding induced by $\dot g$}. Thanks to the argument mentioned just above Lemma \ref{26}, this can be extended to a complete embedding, that we still call $\imath$, from $\mathrm{RO}(\mathbb Q)$ to $\mathrm{RO}(\mathbb P)$.
\end{definition}

As we stated earlier, projections and complete embeddings are closely entwined:

\begin{lemma}\label{karagila}
Let $\mathbb P$ and $\mathbb Q$ be forcing posets. If $\imath:\mathbb Q\to\mathbb P$ is a complete embedding, then there exists a projection $\varrho: \mathbb P\to \mathrm{RO}(\mathbb Q)$ with $\varrho\circ \imath=\id_{\mathbb Q}$, defined by
\[\varrho(p):=\bigvee\{q\in\mathbb Q: \imath(q)\le_{\mathbb P}p\}.\] If $\varrho:\mathbb P\to\mathbb Q$ is a projection, then there is a complete embedding $\imath:\mathbb Q\to\mathrm{RO}(\mathbb P)$:
\[\imath(q):=\bigvee\{p\in\mathbb P: \varrho(p)\le_{\mathbb Q} q\}.\]
Notice that, in this case, the identity $\imath\circ \varrho=\id_{\mathbb P}$ does not generally hold.
\end{lemma}

\begin{proof}
A good proof of this well-known fact is found in \cite[Theorem 2.25]{Karagila}.
\end{proof}

Hence, the existence of a projection between two Boolean algebras is equivalent to the existence of a complete embedding between them but in the other direction.

\begin{definition}
If $\mathbb P$ is a forcing poset and $p\in\mathbb P$, we write
\[\mathbb P_{\downarrow p}:=\{q\in \mathbb P: q\le_{\mathbb P} p\}.\]
\end{definition}

\begin{lemma}\label{fhs}
    Let $\mathbb P$ and $\mathbb Q$ be forcing posets. Let $p\in\mathbb P$ and $\dot g$ be a $\mathbb P$-name with
    \[p\Vdash_{\mathbb P} ``\dot g\ \text{is}\ \hacek{\mathbb Q}\text{-generic}".\]
    Consider the map $\varrho:\mathbb P_{\downarrow p}\to \mathrm{RO}(\mathbb Q)$ defined, for $p'\le_{\mathbb P} p$, as
    \[\varrho(p'):=\bigwedge \{q\in\mathbb Q: p'\Vdash_{\mathbb P} (\hacek q\in \dot g)\}.\]
    Then $\varrho$ is a projection from $\mathbb P_{\downarrow p}$ to $\mathrm{RO}(\mathbb Q)_{\downarrow \varrho(p)}$, as well as its extension to $\mathrm{RO}(\mathbb P)_{\downarrow p}$ that can be constructed by the usual density argument. Moreover, for each $\mathbb P$-generic filter $G$ over $V$ with $p\in G$, the upwards closure in $\mathrm{RO}(\mathbb Q)$ of $\varrho''G$ is given by
    \[\{b\in\mathrm{RO}(\mathbb Q): \exists q\in i_G(\dot g): (q\le b)\}.\]
\end{lemma}

\begin{proof}
    The proof of this lemma can be found in \cite[Lemma 2.1]{FHS}.
\end{proof}

In the situation described in Lemma \ref{fhs}, we refer to the map \[\varrho:\mathrm{RO}(\mathbb P)_{\downarrow p}\to\mathrm{RO}(\mathbb Q)_{\downarrow \varrho(p)}\] as \emph{the projection induced by $\dot g$ and $p$}. We stated Lemma \ref{fhs} primarily to obtain a specific corollary, which is used repeatedly throughout the paper, and is as follows. Let $\pi:\mathbb P\to\mathbb Q$ be a weak projection, and let $\Gamma$ be the standard $\mathbb P$-name for a generic filter, namely,
\[\Gamma:=\{\langle \hacek p,p\rangle: p\in\mathbb P\}.\]
Then, clearly, any $p\in\mathbb P$ forces that $\Gamma$ is $\mathbb P$-generic, and, thanks to Lemma \ref{filterweak}, we also know that any $p\in\mathbb P$ forces that $\pi''\Gamma$ is $\mathbb Q$-generic. Hence, we can consider the projection induced by $\pi''\Gamma$ and $p$, which is of the form
\[\mathrm{IP}(\pi,p):\mathrm{RO}(\mathbb P)_{\downarrow p} \to \mathrm{RO}(\mathbb Q)_{\downarrow \pi(p)}.\]
Note that the \emph{induced projection} $\mathrm{IP}(\pi,p)$ only depends on the projection $\pi:\mathbb P\to\mathbb Q$ and on the choice of a specific element $p\in\mathbb P$, and nothing else, since $\Gamma$ is canonical.

\begin{definition}\label{wposbullet}
    Let $\mathsf{wPos}^\bullet$ denote the concrete, locally small category where the objects are ordered pairs $\langle \mathbb P,p\rangle$, with $\mathbb{P}$ a poset and $p \in \mathbb{P}$, while the morphisms from $\langle\mathbb P,p\rangle$ to $\langle \mathbb Q,q\rangle$ are weak projections $\pi:\mathbb P\to\mathbb Q$ such that $\pi(p)=q$.
\end{definition}

\begin{definition}\label{bcBa}
    Let $\mathsf{bcBa}$ denote the concrete, locally small category where the objects are bottomless complete Boolean algebras and the morphisms are projections between partial orders.
\end{definition}

\begin{lemma}\label{cat}
    Consider the mapping $\Psi:\mathsf{wPos}^\bullet\to\mathsf{bcBa}$ defined as follows. On objects, it acts as \[\Psi(\langle \mathbb P,p\rangle):=\mathrm{RO}(\mathbb P)_{\downarrow p},\]
    while on morphisms, it sends any weak projection $\pi:\langle \mathbb P,p\rangle\to\langle \mathbb Q,\pi(p)\rangle$ to
    \[\Psi(\pi):=\mathrm{IP}(\pi,p):\mathrm{RO}(\mathbb P)_{\downarrow p}\to\mathrm{RO}(\mathbb Q)_{\downarrow \pi(p)}=\Psi(\langle \mathbb P,p\rangle)\to\Psi(\langle \mathbb Q,\pi(p)\rangle).\]
    Thus, $\Psi$ is a functor from the category $\mathsf{wPos}^\bullet$ to the category $\mathsf{bcBa}$. 
\end{lemma}

\begin{proof}This lemma is proved, with different notation, in \cite[Lemma 2.10]{DPT}.
\end{proof}

To wrap up the section, we recall the definition of quotient forcing.

\begin{definition}
    Given a projection $\varrho:\mathbb P\to\mathbb Q$ and a $\mathbb Q$-generic filter $H$ over $V$, we define the \emph{quotient forcing} $\mathbb P/H$ as the subposet of $\mathbb P$ of all $p\in\mathbb P$ with $\varrho(p)\in H$.
\end{definition}

\begin{lemma}
    Every $\mathbb P/H$-generic filter over $V[H]$ is $\mathbb P$-generic over $V$.
\end{lemma}

\begin{proof}
    The proof of this lemma is in \cite[Lemma V.4.45]{Kunen}. There, the quotient is defined from a complete embedding, but it is equivalent to our definition thanks to Lemma \ref{karagila}.
\end{proof}

\section{Sigma-Příkrý Forcing}

To state our Definition \ref{k-Solovay}, the analogue of Definition \ref{Solovay} in the higher context of $L(V_{\kappa+1})$, we first need to discuss $\Sigma$-Příkrý forcing. This concept was introduced by Alejandro Poveda, Assaf Rinot and Dima Sinapova in their article \cite{PRS1}, and further developed in \cite{PRS2} and \cite{PRS3}, where they axiomatised fundamental properties satisfied by the original ``vanilla Příkrý forcing'' defined by Karel Příkrý in \cite{Pri}. \medskip

Before defining what a $\Sigma$-Příkrý forcing is, we need to fix a so-called \emph{$\Sigma$-sequence}: a non-decreasing sequence $\langle \kappa_n: n<\omega\rangle$ of regular uncountable cardinals converging to some strong limit $\kappa$. This will serve as a blanket assumption throughout this article.

\begin{definition}
Given a $\Sigma$-sequence $\langle \kappa_n:n<\omega\rangle$ as before, a forcing poset $\langle \mathbb P, \le, \mathbb 1\rangle$ is \emph{$\Sigma$-Příkrý} if there exist
\begin{itemize}
    \item a surjective map $\ell:\mathbb P \to \omega$, called \emph{length};
    \item a cardinal $\mu$;
    \item a map $c:\mathbb P \to \mu$
\end{itemize}
satisfying the following conditions:
\begin{itemize}
    \item for all $q\le p$, it holds $\ell(q)\ge \ell(p)$; 
    \item for all $p\in\mathbb P$, there is $q\le p$ with $\ell(q)=\ell(p)+1$;
    \item the subposet $\langle \ell^{-1}(\{n\})\cup\{\mathbb 1\}, \le, \mathbb 1\rangle$ is $\kappa_n$-directed-closed, meaning that if $L\subseteq\mathbb P$ is a set of pairwise compatible elements with $|L|<\kappa_n$ and such that $\ell(p)=n$ for all $p\in L$, there is some $r\in\mathbb P$ with $\ell(r)=n$ and $r\le p$ for all $p\in L$;
    \item $\mathbb 1 \Vdash_{\mathbb P} (\check \mu = \check \kappa ^+)$;
    \item for all $p$ and $q$ in $\mathbb P$, if $c(p)=c(q)$ then there is $r\le p$, $r \le q$ such that \[\ell(r)=\ell(p)=\ell(q);\]
    \item for $q\le p$, $n<\omega$ and $m<\omega$, if $\ell(q)=\ell(p)+n+m$, then the set
    \[\{r\le p: q\le r,\ \ell(r)=\ell(p)+n\}\]
    has a greatest element, denoted by $m(p,q)$;
    \item in the particular case $m=0$, we write $w(p,q)$ instead of $0(p,q)$;
    \item for all $p\in\mathbb P$, the set $W(p):=\{w(p,q): q\le p\}$ has size ${<}\mu$;
    \item for $p'\le p$, the map $q\mapsto w(p,q)$ is order-preserving from $W(p')$ to $W(p)$;
    \item if $U\subseteq\mathbb P$ is such that for all $r\in U$ and $q\le r$ if $\ell(q)=\ell(r)$ then $q\in U$, then, for all $p\in\mathbb P$ and $n<\omega$, there is $q\le p$ with $\ell(q)=\ell(p)$ such that
    \[\{r\in \mathbb P : r\le q,\ \ell(r)=\ell(q)+n\}\]
    is either contained in $U$ or disjoint from $U$.
\end{itemize}
\end{definition}

We refer the reader to the articles \cite{PRS1}, \cite{PRS2} and \cite{PRS3} for more details on this kind of forcings and their applications to various branches of set theory. However, we state a couple of important properties that are going to be fundamental in this article. The first, typically referred to as the \emph{strong Příkrý property}, is:


\begin{lemma}[Strong Příkrý Property]\label{strongPříkrýproperty}
    Let $\mathbb P$ be a $\Sigma$-Příkrý poset with $\ell:\mathbb P\to\omega$ as length. For all $p\in\mathbb P$ and all dense open sets $D\subseteq\mathbb P$, there are a condition $q\le p$ with $\ell(q)=\ell(p)$ and $n>0$ such that every $r\le q$ with $\ell(r)\ge\ell(q)+n$ belongs to $D$.
\end{lemma}

\begin{proof}
    The proof of this result can be found in \cite[Corollary 2.7]{PRS1}.
\end{proof}

Another fact about $\Sigma$-Příkrý forcings that is crucial for us is the following.

\begin{lemma}\label{bddsbs}
Forcing with a $\Sigma$-Příkrý poset does not add bounded subsets of $\kappa$.
\end{lemma}

\begin{proof}
The proof of this lemma is contained in \cite[Lemma 2.10]{PRS1}. 
\end{proof}


In order to introduce systems and weak systems of forcing posets, we adopt a category-theoretic approach, which provides a cleaner and more insightful perspective on the fundamental properties of the definition. This is similar to our approach in the statement of Lemma \ref{cat}, in contrast to Lemma 2.10 of \cite{DPT}.

\begin{definition}
    A \emph{posetal category} is a category $\mathsf{D}$ where there is at most one morphism from any one object to another. Such a category can be viewed as a partial order, where $d \le e$ if and only if there exists a morphism $d \to e$. We focus on posetal categories that are \emph{downwards directed}, meaning that for any pair of objects $d$ and $e$ in the category there are an object $f$ and morphisms $f\to d$ and $f\to e$. Equivalently, viewing the category as a partial order, we require that any two elements have a lower bound.
\end{definition}

A \emph{directed diagram} in a category $\mathsf{C}$ is a functor from a downwards directed posetal category $\mathsf{D}$ to $\mathsf{C}$. In this article, we do not consider upwards directed categories, as these are simply the opposites of downwards directed categories.\medskip


An example of a directed diagram in the category $\mathsf{wPos}^\bullet$ is the following:
\begin{center}
    \begin{tikzcd}
    \langle \mathbb P_0, p_0\rangle \arrow[r, "\pi_{0,1}"] \arrow[rd, "\pi_{0,2}"'] & \langle \mathbb P_1, p_1\rangle \arrow[d, "\pi_{1,2}"] \\
    & \langle \mathbb P_2, p_2\rangle
\end{tikzcd}
\end{center}
where the maps $\pi_{0,1}$, $\pi_{0,2}$ and $\pi_{1,2}$ are weak projections such that $p_1=\pi_{0,1}(p_0)$ and $p_2=\pi_{0,2}(p_0)$. One important fact about directed diagrams is that they are preserved by functors. So for example we can apply the functor $\Psi$ from Lemma \ref{cat} and find another directed diagram
\begin{center}
    \begin{tikzcd}
    \mathrm{RO}(\mathbb P_0)_{\downarrow p_0} \arrow[r, "\varrho_{0,1}"] \arrow[rd, "\varrho_{0,2}"'] & \mathrm{RO}(\mathbb P_1)_{\downarrow p_1} \arrow[d, "\varrho_{1,2}"] \\
    & \mathrm{RO}(\mathbb P_2)_{\downarrow p_2}
\end{tikzcd}
\end{center}
recalling that $\mathrm{RO}(\mathbb P_i)_{\downarrow p_i}=\Psi(\langle \mathbb P_i,p_i\rangle)$ and $\varrho_{i,j}:=\mathrm{IP}(\pi_{i,j}, p_i)=\Psi(\pi_{i,j})$.

\begin{definition}\label{def: nice systems}
    A weak directed system of forcing posets or, shortly, a weak system, is a cone over a directed diagram in the category $\mathsf{wPos}^\bullet$.
\end{definition}

Definition \ref{def: nice systems} is equivalent to Definition 3.6 in \cite{DPT}, but is expressed using the language of category theory. For sake of completeness and clarity, now we unravel Definition \ref{def: nice systems} to show what is really a weak system. First, we fix a downwards directed posetal category $\mathsf D$. For each object $d$ of $\mathsf{D}$, we consider a partial order $\mathbb P_d$ and an element $p_d \in \mathbb P_d$. For each arrow $d \to e$ in $\mathsf{D}$, we consider a weak projection $\pi_{d,e} : \mathbb P_d \to \mathbb P_e$ such that $\pi_{d,e}(p_d) = p_e$. This can be considered as a directed diagram in the category $\mathsf{wPos}^\bullet$ if
\begin{itemize}
    \item $\pi_{d,d}=\mathrm{id}_{\mathbb P_d}$ for each object $d$ in $\mathsf D$;
    \item $\pi_{e,f}\circ \pi_{d,e}=\pi_{d,f}$ whenever $d\to e$ and $e\to f$ in $\mathsf D$.
\end{itemize}
Now, a \emph{cone} over this directed diagram consists of a partial order $\mathbb P_\infty$, an element $p_\infty \in \mathbb P_\infty$, and weak projections $\pi_{\infty, d}$ for each object $d$ of $\mathsf{D}$ such that:
\begin{itemize}
\item $\pi_{\infty,d}(p_\infty) = p_d$ for each object $d$ of $\mathsf{D}$;
\item $\pi_{d,e} \circ \pi_{\infty,d} = \pi_{\infty,e}$ for all arrows $d \to e$ in $\mathsf{D}$.
\end{itemize}
When the context is clear, we often abuse notation by saying ``let $\langle \mathbb P_\infty, p_\infty\rangle$ be a weak system" without explicitly writing the morphisms $\pi_{\infty,d}$, or even just ``let $\mathbb P_\infty$ be a weak system", implicitly assuming that $p_\infty = \mathbb 1_{\mathbb P_\infty}$.\medskip

If every $\mathbb P_d$ for $d$ an object of $\mathsf{D}$, as well as $\mathbb P_\infty$, is a $\Sigma$-Příkrý forcing poset, we say that $\mathbb P_\infty$ is a \emph{weak $\Sigma$-system}.

\begin{definition}\label{miau}
    A directed system of Boolean algebras or, shortly, a system, is a cone over a directed diagram in the category $\mathsf{bcBa}$.
\end{definition}

Assume $\langle\mathbb P_\infty, p_\infty\rangle$ is a weak system over some directed diagram. Since such directed diagram is a functor from a downwards directed posetal category $\mathsf D$ to $\mathsf{wPos}^\bullet$, we can compose this functor with the functor $\Psi$ from Lemma \ref{cat}. This way, we get a directed diagram in the category $\mathsf{bcBa}$, and $\Psi(\langle \mathbb P_\infty, p_\infty\rangle)=\mathrm{RO}(\mathbb P_\infty)_{\downarrow p_\infty}$ is a cone over this directed diagram, with the projections given by $\Psi(\pi)=\mathrm{IP}(\pi_{\infty, d}, p_\infty)$. We can always do this, whenever we have a weak system. We call this the \emph{Boolean completion} of the weak system, and write $\mathbb B_d:=\mathrm{RO}(\mathbb P_d)_{\downarrow p_d}$ for all objects $d$ of $\mathsf D$, in order to have less cumbersome notation.\medskip

Another important definition from \cite{DPT} is the following. For the rest of the section, let $\kappa$ be a strong limit cardinal and let $\lambda>\kappa$ be inaccessible.

\begin{definition}\label{nice systems: capturing, bounded, amenable}
    A weak system $\mathbb P_\infty$ is $(\kappa,\lambda)$-nice if it is:
    \begin{itemize}
        \item\label{item 1: capturing} \emph{$\kappa$-capturing}, meaning that if $G\subseteq\mathbb P_\infty$ is a generic filter over $V$, then for every subset $x\subseteq\kappa$ in $V[G]$ and every object $d$ in $\mathsf D$ there is an arrow $e\to d$ such that $x\in V[\pi''_{\infty, e}G]$;
        \item \emph{$\lambda$-bounded}, meaning that $\mathbb P_d\in H(\lambda)$ for all objects $d$ of $\mathsf D$;
        \item \emph{amenable to interpolation}, meaning that whenever $G\subseteq\mathbb P_\infty$ is a generic filter over $V$, then $\lambda=(\kappa^+)^{V[G]}$ and $\cof(\kappa)^{V[G]}=\omega$.
    \end{itemize}
\end{definition}

For our purposes, the $\kappa$-capturing property is not enough, and we need the apparently stronger $({<}\lambda)$-capturing property to hold in order to prove our main theorems. Crucially, and contrary to our original beliefs, that were present in an earlier version of this article, this is directly implied by the $\kappa$-capturing property. A similar behaviour can be noticed in the classical Solovay model:

\begin{lemma}\label{lem: alternative def of Solovay}
    Let $V\subseteq M$ be models of \textsf{ZFC}, with $L(\mathbb R)^M$ Solovay over $V$. Then for all $\alpha<\omega_1^M$ and all $x\subseteq\alpha$ in $M$, $x$ is \emph{small-generic} over $V$: there is a countable poset $\mathbb P$ in $V$ and a filter $g\subseteq\mathbb P$ generic over $V$ with $V[x]\subseteq V[g]$.
\end{lemma}

\begin{proof}
Note that, thanks to Bagaria-Woodin's Theorem \ref{BagariaWoodin}, it is enough to prove this lemma in the case $M=V[C]$ for $C\subseteq\Coll(\omega, {<}\omega_1^M)$ a filter generic over $V$.\medskip

From now onwards the proof will be the same as the original proof of Theorem \ref{welldef} done by Solovay in his \cite[Section 3.4]{Sol70}. There, he proved that for all $y\subseteq\omega$ in $M$, $y$ is small-generic over $V$, but during the proof he showed something more: indeed, that for all such $y$ there is some $\beta<\omega_1^M$ such that, if we call \[\ell_{\lambda,\beta}:\Coll(\omega,{<}\lambda)\to\Coll(\omega,{\le}\beta)\] the projection (like we did in Example \ref{lévy}), then it holds $V[y]\subseteq V[\ell_{\lambda,\beta}''G]$. Notice that in this case $\ell_{\lambda,\beta}''G\subseteq\Coll(\omega,{\le}\beta)$ is the filter that is generic over $V$, and the forcing poset $\Coll(\omega,{\le}\beta)$ is countable in $M$ because $\beta$ itself is countable.\medskip


So, in our case, let $\alpha<\omega_1^M$ and let $x\subseteq\alpha$ be in $M$. Notice that, from the point of view of $M$, $\alpha$ is countable, so there is a bijection $f:\alpha\to\omega$ that is being added by the forcing $\Coll(\omega,{\le}\alpha)$. Let $y:=f''x$, which is a subset of $\omega$ and belongs to $M$.\medskip

Thanks to Solovay's proof of Theorem \ref{welldef}, as we just remarked, there is $\beta<\omega_1^M$ with $V[y]\subseteq V[\ell_{\lambda,\beta}''G]$. If we define $\gamma:=\max\{\alpha,\beta\}$, we have $\gamma<\omega_1^M$, so the partial order $\Coll(\omega,{\le}\gamma)$ is countable in $M$, and if we call $g:=\ell_{\lambda,\gamma}''G\subseteq \Coll(\omega,{\le}\gamma)$, this $g$ is a filter that is $V$-generic, and from $f\in V[g]$ and $y\in V[g]$ we get $x\in V[g]$.
\end{proof}

\begin{lemma}\label{39}
    A $(\kappa,\lambda)$-nice weak system is also $({<}\lambda)$-capturing, meaning that if $G\subseteq\mathbb P_\infty$ is a generic filter over $V$, then for every $\alpha<(\kappa^+)^{V[G]}$, every subset $x\subseteq\alpha$ is any set in $V[G]$ and every object $d$ in $\mathsf D$ there is an arrow $e\to d$ such that $x\in V[\pi''_{\infty, e}G]$
\end{lemma}

\begin{proof}In the situation described by the statement, notice that, if we denote by $\mathrm{trcl}(\{x\})$ the transitive closure of $\{x\}$, this has size at most $\kappa$ in $V[G]$. Hence, there is a bijection $f:\eta\to \mathrm{trcl}(\{x\})$ in $V[G]$ for some $\eta\le\kappa$. Let us now define the following subset of $\eta^2$:
\[E:=\{\langle \alpha,\beta\rangle\in \eta^2: f(\alpha)\in f(\beta)\}.\]
We can code $E$ as a subset of $\kappa$ thanks to a Gödel's pairing function, hence by the $\kappa$-capturing property we know there is an arrow $e\to d$ such that $E\in V[\pi_{\infty,e}''G]$. But the transitive collapse of $\langle \eta,E\rangle$, that can be computed inside $V[\pi_{\infty,e}''G]$, is $\langle \mathrm{trcl}(\{x\}),\in\rangle$ hence we get $x\in V[\pi_{\infty,e}''G]$.
\end{proof}

\section{Examples of Nice Weak Sigma-Systems}

In the previous section, we recalled the notion of $(\kappa,\lambda)$-nice weak $\Sigma$-systems, introduced in \cite{DPT} in order to construct models of \textsf{ZF} satisfying certain regularity properties for subsets of $\kappa$. The aim of this section is to present the two main examples of $(\kappa,\lambda)$-nice weak $\Sigma$-systems.\medskip

Our first example is the \emph{Merimovich supercompact extender-based Příkrý forcing}, introduced by Carmi Merimovich in \cite{Meri}. We show it has the structure of a $(\kappa,\lambda)$-nice weak $\Sigma$-system.




\begin{definition}\label{merimovich}
    Let $\lambda$ be an inaccessible cardinal and $\kappa<\lambda$ a $(<\lambda)$-supercompact cardinal. This means we assume the existence of an elementary embedding $j:V\to M$ such that $\crit(j)=\kappa$, $j(\kappa)=\lambda$, and $M^{<\lambda}\subseteq M$. Let us now consider the downwards directed posetal category $\mathsf D$ where the objects are the elements in the set \[\mathcal D:=\{d\in [\lambda-\kappa]^{<\lambda} : \kappa\in d\}\]
    and we have an arrow $d\to e$ if and only if $d\supseteq e$. This posetal category is downwards directed because for all $d$ and $e$ in $\mathcal D$, it holds that $d\cup e\in \mathcal D$ and also $d\cup e\to d$, $d\cup e\to e$.\medskip
    
    For each $d\in\mathcal D$, we call $d$-\emph{object} any map $\nu: \dom(\nu)\to\kappa$ such that
    \begin{itemize}
        \item $\kappa\in\dom(\nu)\subseteq d$;
        \item $\left|\dom(\nu)\right|<\kappa$;
        \item if $\alpha<\beta$ are in $\dom(\nu)$, then $\nu(\alpha)<\nu(\beta)$.
    \end{itemize}
The set of $d$-objects is called $\Ob(d)$. Given $\nu$ and $\mu$ in $\Ob(d)$ we say that $\nu\ll \mu$ if
\begin{itemize}
    \item $\dom(\nu)\subseteq\dom(\mu)$;
    \item $|\nu|<\mu(\kappa)$;
    \item for all $\alpha\in \dom(\nu)$,  it holds $\nu(\alpha)<\mu(\kappa)$.
\end{itemize}
Now, for $d\in\mathcal D$, let us define the \emph{maximal coordinate} of $d$ as
\[\mc(d):=\{\langle j(\alpha),\alpha\rangle: \alpha\in d\}.\]
For $d\in \mathcal D$, the maximal coordinate induces a $\kappa$-complete ultrafilter over $\Ob(d)$:
\[E(d):=\{X\subseteq\Ob(d):\mc(d)\in j(X)\}.\]
We define the \emph{Merimovich supercompact extender-based Příkrý forcing} (or, just, Merimovich forcing) as the poset $\langle \mathbb P_\infty, \le_\infty\rangle$ where the conditions in $\mathbb P_\infty$ are pairs $\langle f,A\rangle$ such that
\begin{itemize}
    \item $\dom(f)\in \mathcal D$;
    \item $f:\dom(f)\to \kappa^{<\omega}$, and $f(\alpha)$ is order-increasing for all $\alpha\in \dom(f)$;
    \item $A\in E(\dom(f))$ and for each $\nu\in A$, $\nu(\kappa)>\max(f(\alpha))$.
\end{itemize}
Given any two conditions $\langle f,A\rangle$ and $\langle g,B\rangle$ in $\mathbb P_\infty$, we say that $\langle f,A\rangle\le^\circ \langle g,B\rangle$ if
\begin{itemize}
    \item $f\subseteq g$;
    \item for every object $\nu\in B$, it holds $\nu\upharpoonright \dom(f) \in A$.
    \end{itemize}
For $p=\langle f,A\rangle\in \mathbb P_\infty$ and $\nu\in A$, we will call $p^\curvearrowright \nu$ the condition $\langle f_\nu, A_\nu\rangle$ where
\begin{itemize}
    \item $f_\nu$ is the function with domain $\dom(f_\nu)=\dom(f)$ and such that
    \[f_\nu(\alpha)=\begin{cases} f(\alpha)& \text{if}\ \alpha\not\in\dom(\nu);\\
    f(\alpha)^\frown\langle \nu(\alpha)\rangle& \text{if}\ \alpha\in\dom(\nu);\end{cases}\]
    \item $A_\nu:=\{\mu\in A: \nu\ll\mu\}.$
\end{itemize}
If $p=\langle f,A\rangle\in\mathbb P_\infty$ and $\vec\nu\in A^{<\omega}$ is a $\ll$-increasing sequence, we define $p^\curvearrowright \vec\nu$ as
\[p^\curvearrowright \vec\nu:=(\cdots((p^\curvearrowright \vec\nu_0)^\curvearrowright \vec\nu_1)^\curvearrowright\cdots\vec\nu_{|\vec\nu|-1}).\]
Now the ordering $\le_\infty$ of the forcing poset $\mathbb P_\infty$ is defined as follows: $\langle f,A\rangle\le_\infty\langle g,B\rangle$ if and only if there is a $\ll$-increasing sequence $\vec\nu\in B^{<\omega}$ with $\langle f,A\rangle\le^\circ \langle g,B\rangle^\curvearrowright \vec\nu$.\medskip

In addition, for each $d\in\mathcal D$, the poset $\mathbb P_d$ will be the subposet of $\mathbb P_\infty$ consisting of all elements $\langle f,A\rangle$ with $\dom(f)\subseteq d$, and the map $\pi_{e,d}:\mathbb P_e\to\mathbb P_d$, for $d\preccurlyeq e$, is defined as \[\pi_{e,d}(\langle f,A\rangle):=\langle f\upharpoonright d, A\upharpoonright d\rangle.\]
\end{definition}


\begin{lemma}\label{merinonfully}
    If $\kappa<\lambda$ are such that $\kappa$ is $(<\lambda)$-supercompact and $\lambda$ is inaccessible, the Merimovich forcing is a $(\kappa,\lambda)$-nice weak $\Sigma$-system.
\end{lemma}

\begin{proof}
    The proof of this fact is in \cite[Section 4.1]{DPT}.
\end{proof}

In \cite{DPT}, the authors introduce a diagonal version of the above forcing, designed by Poveda and the second author of this paper during a visit of the latter to Harvard in the spring of 2024.

\begin{definition}\label{meridiag}
Let $\lambda$ be an inaccessible cardinal and $\kappa=\sup_{n<\omega}\kappa_n$, where each of the cardinals $\kappa_n$ is $({<}\lambda)$-supercompact, so for each $n<\omega$ there is some $j_n: V\to M_n$ with $\crit(j_n)=\kappa_n$, $j_n(\kappa_n)>\lambda$ and $M_n^{<\lambda}\subseteq M_n$. We now define, for $n<\omega$,
    \[\mathcal D_n:=\{d \in [\{\kappa_n\}\cup(\lambda-\kappa)]^{<\lambda}: \kappa_n\in d\}.\]
    Given $d\in \mathcal D_n$, we call $d$-object any map $\nu: \dom(\nu)\to \kappa_n$ such that:
    \begin{itemize}
        \item $\kappa_n\in \dom(\nu)\subseteq d$;
        \item $\left|\dom(\nu)\right|<\kappa_n$;
        \item $\nu(\kappa_{n})$ is an inaccessible cardinal above $\kappa_{n-1}$;
        \item if $\alpha<\beta$ are in $\dom(\nu)$, then $\nu(\alpha)<\nu(\beta)$.
    \end{itemize}
    The set of $d$-objects is denoted by $\Ob_n(d)$, and notice that the $n$ can be calculated from $d$ since $\kappa_n=\min(d)$. For $d\in \mathcal D_n$, define the \emph{maximal coordinate} of $d$ as
    \[\mc_n(d):=\{\langle j_n(\alpha),\alpha\rangle:\alpha\in d\}.\]
    For $d\in \mathcal D_n$, this induces a $\kappa_n$-complete ultrafilter over $\Ob(d)$:
    \[E_n(d):=\{X\subseteq \Ob_n(d): \mc_n(d)\in j_n(X)\}.\]
    Clearly, $E_n(d)$ is a $\kappa_n$-complete (yet, not necessarily normal) measure on $\mathrm{Ob}_n(d)$. Since $\kappa_n\in \dom(\nu)$ for any $\nu\in\mathrm{Ob}_n(d)$, it is nevertheless the case that $E_n(d)$ projects onto a normal measure; specifically, it projects to \[\mathcal{U}_n:=\{X\subseteq \kappa_n\mid \kappa_n\in j_n(X)\}\]
via the map sending $\nu$ to $\nu(\kappa_n)$.\medskip

Next, let $n\le m$ and let $d_n\in\mathcal D_n$ and $d_m\in\mathcal D_m$, respectively. We write  $d_n\subseteq^\ast d_m$ whenever $d_n\setminus\{\kappa_n\}\subseteq d_m$. Now, assume that $d_n\subseteq^\ast d_m$. Given $\nu\in\mathrm{Ob}_n(d_n)$ and $\mu\in\mathrm{Ob}_m(d_m)$, we write $\nu\ll \mu$ whenever
\begin{itemize}
    \item $\dom(\nu)\subseteq^\ast \dom(\mu)$;
    \item $\left|\nu\right|<\mu(\kappa_m)$;
    \item $\nu(\alpha)<\mu(\kappa_m)$ for all $\alpha\in\dom(\nu)$.
\end{itemize}

This definition yields a transitive order between objects. We are now ready to define the \emph{diagonal version of the Merimovich supercompact extender-based Příkrý forcing} (or, just, diagonal Merimovich forcing) as follows. A condition in $\mathbb P_\infty$ will be a sequence of length $\omega$ of the form
\[p=\langle f^p_0,\dots, f^p_{\ell(p)-1},\langle f^p_{\ell(p)}, A^p_{\ell(p)}\rangle,\langle f^p_{\ell(p)+1},A^p_{\ell(p)+1}\rangle\dots\rangle\]
satisfying the following conditions:
\begin{itemize}
    \item for each $n<\omega$, $f_n^p$ is a partial function from $\dom(f^p_n)\subseteq\lambda$ to $\kappa_n$ with $\dom(f^p_n)\in \mathcal D_n$ and $f^p_n(\kappa_n)>\kappa_{n-1}$ is an inaccessible cardinal;
    \item for each $n\ge \ell(p)$, $A_n^p\in E_n(\dom(f^p_n))$ and the sequence
    \[\langle \dom(f^p_n): n\ge \ell(p)\rangle\]
    is $\subseteq^\ast$-increasing.
\end{itemize}
Given two conditions $p$ and $q$ in $\mathbb P_\infty$, we say that $p\le^\circ q$ if and only if
\begin{itemize}
    \item $\ell(p)=\ell(q)$;
    \item $f^q_n\subseteq f^p_n$ for all $n<\omega$;
    \item for all $n\ge\ell(p)$, and all $\nu\in A^p_n$, it holds that
    \[(\nu\upharpoonright \mathrm{dom}(f^q_n))\in A^q_n.\]
\end{itemize}

Now, let $p\in\mathbb P_\infty$ and $\nu\in A^p_{\ell(p)}$. We define $p\cat\nu$ as the sequence
\[\langle f^p_0,\dots, f^p_{\ell(p)-1},f^p_{\ell(p)}\oplus \nu, \langle f^p_{\ell(p)+1}, (A^p_{\ell(p)+1})_{\langle\nu\rangle}\rangle,\langle f^p_{\ell(p)+2}, (A^p_{\ell(p)+2})_{\langle\nu\rangle}\rangle\dots\rangle\]
 where $f^p_{\ell(p)}\oplus \nu$ and $(A^p_n)_{\langle \nu\rangle}$ for $n>\ell(p)$ are defined as follows:
    \begin{itemize}
        \item $f^p_{\ell(p)}\oplus \nu$ is the function with domain $\dom(f^p_{\ell(p)})$ and values
        $$(f^p_{\ell(p)}\oplus \nu)(\alpha):=\begin{cases}
            \nu(\alpha) & \text{if}\ \alpha\in\dom(\nu);\\
            f^p_{\ell(p)}(\alpha) & \text{if}\ \alpha\not\in\dom(\nu).
        \end{cases}
        $$
        \item $(A^p_n)_{\langle \nu\rangle}:=\{\eta\in A^p_n\mid \nu\prec \eta\}$ for each $n>\ell(p)$.
    \end{itemize}
More generally, given a finite $\ll$-increasing sequence
\[\vec\nu=\langle \nu_{\ell(p)},\dots,\nu_{m-1}\rangle\]
where each $\nu_n$ belongs to $A^p_n$, one defines $p\cat\vec\nu$ as \[p\cat\vec\nu=(\cdots((p^\curvearrowright \vec\nu_{\ell(p)})^\curvearrowright \vec\nu_{\ell(p)+1})^\curvearrowright\cdots\vec\nu_{m-1}.\]

Now we are ready to define the ordering $\le_\infty$ on $\mathbb P_\infty$. We say that $p\le q$ if and only if $\ell(q)\le \ell(p)$ and there is a $\ll$-increasing sequence $\vec\nu$ such that $p\le^\circ q^\curvearrowright\vec\nu$. Next, we define the weak system structure of $\mathbb P_\infty$. Consider the downwards directed posetal category $\mathsf D$ where objects are $\subseteq^\ast$-increasing sequences $\langle d_n:n<\omega\rangle$ where each $d_n\in \mathcal D_n$. Given two objects $d$ and $e$, we have an arrow $e\to d$ if and only if $d_n\subseteq^\ast e_n$ for each $n<\omega$. This is downwards directed, because given any two objects $d$ and $e$ we can consider $f:=\langle d_n\cup e_n: n<\omega\rangle$, that is still an object of $\mathsf D$, and is such that $f\to e$ and $f\to d$ since $d_n\subseteq^\ast d_n\cup e_n$ and $e_n\subseteq^\ast d_n\cup e_n$ for each $n<\omega$.\medskip

For each object $d$ in the category $\mathsf D$, we define $\mathbb P_d$ as the subposet of $\mathbb P_\infty$ consisting of all elements $p$ such that $\dom(f^p_n)\subseteq d_n$ for each $n<\omega$.
\end{definition}

\begin{lemma}
    If $\lambda$ is inaccessible and $\kappa=\sup_{n<\omega}\kappa_n <\lambda$, where each of the $\kappa_n$ is $({<}\lambda)$-supercompact, the diagonal Merimovich forcing is a $(\kappa,\lambda)$-nice weak $\Sigma$-system.
\end{lemma}

\begin{proof}
    The proof of this fact is in \cite[Section 4.2]{DPT}.
\end{proof}
\section{Axiomatising Higher Solovay Models}\label{sec: Higher Solovay Models}
\color{black}

Now we are finally ready to define when a model $L(V_{\kappa+1})^M$ is a $\kappa$-Solovay model over $V$, following the footsteps of Definition \ref{Solovay}. Recall that, in this context, $\kappa$ is a fixed strong limit cardinal, supremum of a non-decreasing $\Sigma$-sequence $\langle \kappa_n:n<\omega\rangle$.

\begin{definition}\label{k-Solovay}
Let $V\subseteq M$ be models of \textsf{ZFC}. We say that the model $L(V_{\kappa+1})^M$ is a \emph{$\kappa$-Solovay model} over $V$ if $M$ satisfies the following conditions:
\begin{itemize}
\item for all subsets $x\subseteq\kappa$, the ordinal $(\kappa^+)^M$ is an inaccessible cardinal in $V[x]$;
\item for $x\subseteq \kappa$ in $M$, there is a $\Sigma$-Příkrý forcing $\mathbb P\in V$ with $|\mathbb P|^M\le\kappa$, and $g\subseteq\mathbb P$ generic over $V$ with $V[x]\subseteq V[g]$ and $\cof(\kappa)^{V[g]}=\omega$.
\end{itemize}
\end{definition}

Observe that, in this context,  the initial segment of the universe up to $\kappa$ cannot change:

\begin{lemma}
    Let $V\subseteq M$ be models of \textsf{ZFC}, with $L(V_{\kappa+1})^M$ $\kappa$-Solovay over $V$. Then $V_\kappa=V_\kappa^M$.
\end{lemma}

\begin{proof}
First we verify that $\wp^M(\alpha)= \wp(\alpha)$, for all $\alpha<\kappa$. If $a\in \wp^M(\alpha)$, then by Definition \ref{k-Solovay} there is a $\Sigma$-Příkrý forcing $\mathbb P\in V$, and $g\subseteq\mathbb P$ generic over $V$ with $a\in V[g]$. Lemma \ref{bddsbs} ensures that $\mathbb P$ does not add any new bounded subset of $\kappa$, yielding $a\in V$. In particular, $\kappa$ is strong limit in $M$. Now let $x\in V_\kappa^M$. Since $\kappa$ is a strong limit in $M$, we get $|\mathrm{trcl}(\{x\})|^M<\kappa$, and so there is a bijection in $M$ from $\mathrm{trcl}(\{x\})$ to some $\alpha<\kappa$. Let $E$ be the relation on $\alpha$ such that \[\langle \alpha, E\rangle \simeq \langle \mathrm{trcl}(\{x\}), \in\rangle.\] Since $E$ is canonically codeable as a subset of $\alpha$ and $\wp^M(\alpha)= \wp(\alpha)$, we deduce that $E\in V$. The Mostowski collapse of $\langle \mu, E\rangle$ computed in $V$ is $\langle \mathrm{trcl}(\{x\}), \in\rangle\in V$, hence $x\in V$.\end{proof}

As we already mentioned, in the article \cite{DPT} Dimonte, Poveda and the second author prove some regularity properties for the $\kappa$-reals of  $L(V_{\kappa+1})^{V[G]}$ for $G\subseteq\mathbb P_\infty$ generic over $V$, where $\mathbb P_\infty$ is a $(\kappa,\lambda)$-nice weak $\Sigma$-system. Our first new result, the higher analogue of Theorem \ref{welldef}, states that the models studied in \cite{DPT} are $\kappa$-Solovay model.

\begin{theorem}\label{k-welldef}
    Let $\mathbb P_\infty$ be a $(\kappa,\lambda)$-nice weak $\Sigma$-system, and let $G\subseteq \mathbb P_\infty$ be a generic filter over $V$. Then $L(V_{\kappa+1})^{V[G]}$ is $\kappa$-Solovay over $V$.
\end{theorem}

\begin{proof}
    We verify that the conditions in Definition \ref{k-Solovay} are met. First, we have to prove that if $x\subseteq\kappa$ belongs to $V[G]$, then the ordinal $(\kappa^+)^{V[G]}$ is inaccessible in $V[x]$. Fix $x\subseteq\kappa$. Since $\mathbb P_\infty$ is $\kappa$-capturing, there is an object $e$ in $\mathsf D$ such that $x\in V[\pi''_{\infty,e}G]$. Now, since $\mathbb P_\infty$ is $\lambda$-bounded, we have $\mathbb P_e \in H(\lambda)$, and since $\mathbb P_e$ cannot affect the cardinal structure above its size, then $\mathbb P_e$ forces that ``$\lambda$ is inaccessible''. But $\mathbb P_\infty$ is amenable to interpolations, hence $\lambda=(\kappa^+)^{V[G]}$, so we get that $(\kappa^+)^{V[G]}$ is inaccessible in $V[\pi_{\infty,e}''G]$, and, since $x \in V[\pi_{\infty,e}''G]$, $(\kappa^+)^{V[G]}$ is inaccessible in $V[x]$.\medskip

Now we prove that if $x\subseteq\alpha$, there is a $\Sigma$-Příkrý $\mathbb P\in V$ with $|\mathbb P|^{V[G]}\le\kappa$, and a $g\subseteq\mathbb P$ generic over $V$ such that $V[x]\subseteq V[g]$ and $\cof(\kappa)^{V[g]}=\omega$. Since $\mathbb P_\infty$ is amenable to interpolations, $\cof(\kappa)^{V[G]}=\omega$. In particular, in $V[G]$ there is a countable sequence that is cofinal in $\kappa$, and such a sequence is a subset of $\kappa$. Hence, by the $\kappa$-capturing property, there is some object $d$ in $\mathsf D$ such that in the extension $V[\pi''_{\infty,d}G]$ the cardinal $\kappa$ has cofinality $\omega$, and crucially $x\in V[\pi''_{\infty,d}G]$. 
  Notice also that, since $\mathbb P_\infty$ is $\lambda$-bounded, then $\mathbb P_d$ has cardinality less than $(\kappa^+)^{V[G]}$, hence $|\mathbb P_d|^{V[G]}\le \kappa$. 
  \end{proof}

\begin{corollary}
Let $G$ be a $V$-generic filter for the Merimovich forcing relative to $\kappa<\lambda$, introduced in Definition \ref{merimovich}. Then $L(V_{\kappa+1})^{V[G]}$ is a $\kappa$-Solovay model over $V$.
\end{corollary}

\begin{corollary}
Let $G$ be a $V$-generic filter for the diagonal Merimovich forcing relative to $\kappa<\lambda$, introduced in Definition \ref{meridiag}. Then $L(V_{\kappa+1})^{V[G]}$ is a $\kappa$-Solovay model over $V$.
\end{corollary}

We end this section by proving a couple of properties of $\kappa$-Solovay models. In order to do that, recall the Trace Lemma from the article \cite{DPT}:

\begin{lemma}\label{OGtrace}
Let $\mathbb P$ be a $\Sigma$-Příkrý forcing and let $G\subseteq\mathbb P$ be generic over $V$. Suppose that in $V[G]$ it holds $\cof(\kappa)=\omega$, and $f:\kappa\to E$ is a function in $V[G]$ with $E\in V$. Then there is a sequence $\langle B_n:n<\omega\rangle\in V[G]$ consisting of bounded subsets of $\kappa$ (each $B_n\in V$ by Lemma \ref{bddsbs}) with $\bigcup_{n<\omega} B_n=\kappa$ and $(f\upharpoonright B_n)\in V$ for all $n<\omega$.
\end{lemma}

\begin{proof}
    The proof of this lemma can be found in \cite[Lemma 3.3]{DPT}.
\end{proof}

We can prove a version of this lemma for the abstract definition of $\kappa$-Solovay models:

\begin{lemma}[Trace Lemma]\label{trace}
Suppose $V\subseteq M$ are models of $\mathsf{ZFC}$ and assume that $L(V_{\kappa+1})^M$ is a $\kappa$-Solovay model over $V$. Let $f:\kappa\to E$ be any function in $M$, such that $E\in V$ and $|E|<(\kappa^+)^M$. There is a sequence $\langle B_n:n<\omega\rangle\in M$ consisting of bounded subsets of $\kappa$ such that $\bigcup_n B_n=\kappa$ and $(f\upharpoonright B_n)\in V$ for all $n<\omega$.
\end{lemma}

\begin{proof}
First, notice that if $E\in V$ is such that $\left|E\right|<(\kappa^+)^M$, then there is some cardinal $\alpha<(\kappa^+)^M$ such that $E$ and $\alpha$ are in bijection, and the bijection $b:E\to \alpha$ is in $V$. So it is enough to prove the theorem for $b\circ f$ in place of $f$ (and then apply $b^{-1}$). Namely, it is enough to prove the theorem when $E$ is some $\alpha<(\kappa^+)^M$.\medskip

In this case, notice that a map $f:\kappa\to\alpha$ can be seen as a subset $f\subseteq\alpha\times\alpha$, and by using the Gödel pairing function we can encode in a definable way $f$ as a subset of $\alpha$. Thanks to Definition \ref{k-Solovay}, there is a $\Sigma$-Příkrý forcing $\mathbb P$ such that $|\mathbb P|^M\le \kappa$ and a filter $g\subseteq\mathbb P$ that is generic over $V$ with $V[f]\subseteq V[g]$, and also $\cof(\kappa)^{V[g]}=\omega$.\medskip

Now we apply Lemma \ref{OGtrace} to this $\Sigma$-Příkrý forcing poset and to the map $f:\kappa\to\alpha$ in $V[g]$, finding a sequence $\langle B_n:n<\omega\rangle\in V[g]\subseteq M$ consisting of bounded subsets of $\kappa$ such that $\bigcup_n B_n=\kappa$ and $(f\upharpoonright B_n)\in V$ for all $n<\omega$, as desired.
\end{proof}

\begin{lemma}[Interpolation Lemma]\label{lemma: interpolation}
    Let $V\subseteq M$ be models of \textsf{ZFC} and $L(V_{\kappa+1})^M$ be a $\kappa$-Solovay model over $V$. Let $\mathbb{P}_1\in V$ be a $\Sigma$-Příkrý poset with $|\mathbb{P}_1|< (\kappa^+)^M$, and $\mathbb{P}_2\in V$ be any forcing poset (not necessarily $\Sigma$-Příkrý, and possibly trivial).\medskip
    
    Let $\pi:\mathbb P_1\to\mathbb P_2$ be a weak projection and let $\varrho:\mathbb B_1\to\mathbb B_2$ denote the image of $\pi$ applied to the pointed poset $\langle\mathbb P_1,\mathbb 1_{\mathbb P_1}\rangle$ under the functor $\Psi$ defined in Lemma \ref{cat}: hence, $\varrho$ is a projection between bottomless complete Boolean algebras and
    \[\mathbb B_1:=\mathrm{RO}(\mathbb P_1);\hspace{.5cm}\mathbb B_2:=\mathrm{RO}(\mathbb P_2)_{\downarrow \pi(\mathbb 1_{\mathbb P_1})}.\]
Let $g$ be $\mathbb B_2$-generic over $V$ with $g\in M$, and let $q\in \mathbb B_1/g$. Then there is a sequence $\langle p_n:n<\omega\rangle\in M$ of elements of $\mathbb P_1$ decreasing in the ordering  $\le_{\mathbb B_1}$, with $p_n\in \mathbb B_1/g$, $p_n\le_{\mathbb B_1} q$ and $\ell(p_n)<\ell(p_{n+1})$, where $\ell$ is the length of $\mathbb P_1$, for all $n<\omega$, and
\[h:=\{b\in\mathbb B_1: \exists n<\omega: p_n\le_{\mathbb B_1} b\}\]
is in $M$ and is $\mathbb B_1/g$-generic over $V$ with $q\in h$.
%
    %


\end{lemma}

\begin{proof}As $|\mathbb{P}_1|<(\kappa^+)^M$ and $(\kappa^+)^M$ is inaccessible, ${|\wp(\mathbb{P}_1)|}<(\kappa^+)^M$. Consider
\[E:=\{D\in\wp(\mathbb P_1)^V: D\ \text{is dense open in}\ \mathbb P_1\}.\]
This set has size at most $\kappa$ in $M$, hence there is $D:=\langle D_\alpha: \alpha<\kappa\rangle\in M$ that is an enumeration of it. This $D$ is a map from $\kappa$ to $E$, hence we can apply the Trace Lemma \ref{trace} to find a sequence $\langle B_n:n<\omega\rangle\in M$ consisting of bounded subsets of $\kappa$ such that $\bigcup_n B_n=\kappa$ and $(D\upharpoonright B_n)\in V$ for all $n<\omega$. Let us now build a sequence $\langle p_n:n<\omega\rangle$ in $M$ such that $p_{n+1}\le_{\mathbb B_1} p_n\le_{\mathbb B_1} q$ and $\ell(p_{n+1})>\ell(p_n)$ for $n<\omega$.\medskip

We proceed by induction. For $n=0$, let us define
\begin{eqnarray*}R_0:=\left\{r\in\mathbb P_1\cap (\mathbb B_1)_{\downarrow q}: \forall \alpha\in B_0\ \exists k_\alpha<\omega\right.\\\left.\forall p\in(\mathbb P_1)_{\downarrow r} \ (\ell(p)\ge\ell(r)+k_\alpha\to p\in D_\alpha)\right\}.\end{eqnarray*}
Note first that $R_0\in V$ because $(D\upharpoonright B_0)\in V$: this is crucial to run the forthcoming argument. Let us prove that $R_0$ is a dense subset of $(\mathbb B_1)_{\downarrow q}$: to do so, fix $b\in (\mathbb B_1)_{\downarrow q}$. Since $B_0$ is a bounded subset of $\kappa$, there is $m<\omega$ with $B_0\subseteq \kappa_m$. Let $\langle \alpha_\beta:\beta<\gamma\rangle$ be an injective enumeration of $B_0$ and let us pick some $r\le_{\mathbb B_1} b$ with $\ell(r)>m$.\medskip

Now, we want to build a decreasing sequence $\langle u_\beta: \beta<\gamma\rangle$ below $r$ in the ordering $\le_{\mathbb B_1}$ such that for each $\beta<\gamma$ there is a natural number $k_{\alpha_\beta}<\omega$ such that for each $p\le_{\mathbb B_1} u_\beta$ with $\ell(p)\ge\ell(u_\beta)+k_{\alpha_\beta}$ it holds $p\in D_{\alpha_\beta}$. This can be done by utilising the Strong Příkrý Property \ref{strongPříkrýproperty} and the $\kappa_{\ell(r)}$-directed-closedness of the subposet $\langle \ell^{-1}(\{\ell(r)\})\cup\{\mathbb 1_{\mathbb P_1}\}, \le_{\mathbb P_1}, \mathbb 1_{\mathbb P_1}\rangle$, by keeping also $\ell(u_\beta)=\ell(r)$ for all $\beta<\gamma$.\medskip

Finally, if we let $r'\le_{\mathbb B_1} u_\beta$ for all $\beta<\gamma$, again using the $\kappa_{\ell(r)}$-directed-closedness of the aforementioned subposet, we get that $r'\le_{\mathbb B_1} b$ and $r'\in R_0$, proving our claim.\medskip

Now that we know $R_0$ is dense in $(\mathbb B_1)_{\downarrow q}$, and since $\varrho$ is a projection, $\varrho''R_0$ is a dense subset of $(\mathbb B_2)_{\downarrow \varrho(q)}$. We picked $q\in\mathbb B_1/g$, hence $\varrho(q)\in g$ and so $\varrho''R_0\cap g\neq\varnothing$. We now can pick a condition $p_0\in R_0$ such that $\varrho(p_0)\in \varrho''R_0\cap g$, as we wanted.\medskip

Now suppose we already defined $p_m$ for $m\le n$, some $n<\omega$. Define now
\begin{eqnarray*}R_{n+1}:=\left\{r\in(\mathbb P_1)_{\downarrow p_n}: \ell(r)>\ell(p_n), \forall \alpha\in B_{n+1}\ \exists k_\alpha<\omega\right.\\\left.\forall p\in(\mathbb P_1)_{\downarrow r} \ (\ell(p)\ge\ell(r)+k_\alpha\to p\in D_\alpha)\right\}.\end{eqnarray*}
We could carry out the same argument as before to show $R_{n+1}$ is dense in $(\mathbb B_1)_{\downarrow p_n}$. Therefore, we can find a $p_{n+1}\in R_{n+1}$ such that $\varrho(p_{n+1})\in \varrho''R_{n+1}\cap g$.\medskip

Thus, we constructed a sequence $\langle p_n:n<\omega\rangle\in M$, decreasing in the ordering $\le_{\mathbb B_1}$, with $\ell(p_{n+1})>\ell(p_n)$, $p_n\in R_n$, $p_n\le_{\mathbb B_1} q$ and $\varrho(p_n)\in g$ for all $n<\omega$. Let
\[h:=\{b\in\mathbb B_1: \exists n<\omega: p_n\le_{\mathbb B_1} b\}.\]
This $h\in M$. To finish the proof, we shall show that $h$ is $\mathbb B_1/g$-generic over $V[g]$.\medskip

By construction, $p_n\in \mathbb B_1/g$ for all $n<\omega$, so it is enough to show $h$ is $\mathbb B_1$-generic over $V$. Let $D'\in V$ be dense open in $\mathbb B_1$ and consider the set $D^\ast\subseteq\mathbb P_1$ defined as
\[D^\ast:=(D'\cap (\mathbb P_1)_{\downarrow q})\sqcup\{q'\in\mathbb P_1: q\ \text{and}\ q'\ \text{are incompatible}\}.\]
This $D^\ast$ is a dense open subset of $\mathbb P_1$, hence there is some $\alpha<\kappa$ such that $D^\ast=D_\alpha$. Let $n<\omega$ be such that $\alpha\in B_n$. By definition of $R_n$, there is some $k_\alpha<\omega$ such that every $p\le_{\mathbb P_1} p_n$ with $\ell(p)\ge \ell(p_n)+\kappa_\alpha$ belongs to $D_\alpha$. By construction, since the lengths of the elements of the sequence $\langle p_n:n<\omega\rangle$ are always increasing there is some $m<\omega$ such that $\ell(p_m) \ge \ell(p_n)+k_\alpha$, yielding $p_m\in D_\alpha=D^\ast$. Clearly $p_m\le_{\mathbb B_1} q$, meaning that $p_m\in D'$. Also $p_m\in h$, hence witnessing $h\cap D'\neq\varnothing$.
\end{proof}

\begin{lemma}[Capturing Lemma]\label{nocap}
    Assume we are in the hypotheses of the Interpolation Lemma \ref{lemma: interpolation}, and there are also another forcing poset $\mathbb P_3\in V$, a $\mathbb P_3$-generic filter $g_3\in M$, and in $V[g]$ there are an element $t\in \mathbb B_1/g$ with $t\le q$ and a projection 
  \[\sigma: (\mathbb B_1/g)_{\downarrow t} \to \mathbb B_3:=\mathrm{RO}(\mathbb P_3).\]
    Then we can pick $h$ as in the Interpolation Lemma \ref{lemma: interpolation} with $t\in h$ and $\sigma''(h_{\downarrow t})=g_3$.
%
\end{lemma}

\begin{proof}
We mimic the construction of the proof of the Interpolation Lemma \ref{lemma: interpolation}, by also ensuring that the $\mathbb B_1/g$-generic $h$ captures the given $\mathbb P_3$-generic $g_3$. Let
\begin{eqnarray*}F_0:=\left\{u\in\mathbb P_1\cap (\mathbb B_1/g)_{\downarrow t}: \forall \alpha\in B_0\ \exists n_\alpha<\omega\right.\\\left.\forall p\in(\mathbb P_1)_{\downarrow u} \ (\ell(p)\ge\ell(u)+n_\alpha\to p\in D_\alpha)\right\}.\end{eqnarray*}
We want to show that $F_0$ is a dense subset of $(\mathbb B_1/g)_{\downarrow t}$. Let $s\in(\mathbb B_1/g)_{\downarrow t}$ and
\begin{eqnarray*}S_0:=\left\{r\in\mathbb P_1\cap (\mathbb B_1)_{\downarrow s}: \forall \alpha\in B_0\ \exists k_\alpha<\omega\right.\\\left.\forall p\in(\mathbb P_1)_{\downarrow r} \ (\ell(p)\ge\ell(r)+k_\alpha\to p\in D_\alpha)\right\}.\end{eqnarray*}
This set $S_0$ is a dense subset of $(\mathbb B_1)_{\downarrow s}$, because its definition is the same as the one of $R_0$ in the Interpolation Lemma \ref{lemma: interpolation}, except on $(\mathbb B_1)_{\downarrow s}$ instead of $(\mathbb B_1)_{\downarrow q}$. Since the map $\varrho:\mathbb B_1\to\mathbb B_2$ is a projection, $\varrho''S_0$ is dense in $(\mathbb B_2)_{\varrho(s)}$. But $\varrho(s)\in g$, so $g\cap \varrho''S_0\neq\varnothing$. Pick $s'\in S_0$ with $\varrho(s')\in g$. Then $s'\le_{\mathbb B_1} s$ and $s'\in F_0$.\medskip

We proved that $F_0$ is dense in $(\mathbb B_1/g)_{\downarrow t}$. Since $\sigma:(\mathbb B_1/g)_{\downarrow t}\to\mathbb B_3$ is a projection, $\sigma''F_0$ is a dense subset of $\mathbb B_3$, hence $g_3\cap \sigma''F_0\neq\varnothing$. Pick $t_0\in F_0$ with $\sigma(t_0)\in g_3$. It is clear by construction that $t_0\le_{\mathbb B_1} t$ and $\varrho(t_0)\in g$. In general, we define a sequence $\langle t_n:n<\omega\rangle$ in $M$ that is $\le_{\mathbb B_1}$-decreasing and such that $\sigma(t_n)\in g_3$ for all $n<\omega$. This will be done following the previous argument taking as dense sets
\begin{eqnarray*}F_{n+1}:=\left\{u\in\mathbb P_1\cap (\mathbb B_1/g)_{\downarrow t_n}: \forall \alpha\in B_0\ \exists n_\alpha<\omega\right.\\\left.\forall p\in(\mathbb P_1)_{\downarrow u} \ (\ell(p)\ge\ell(u)+n_\alpha\to p\in D_\alpha)\right\}.\end{eqnarray*}
As in the proof of the Interpolation Lemma \ref{lemma: interpolation}, we can show that
\[h:=\{b\in\mathbb B_1: \exists n<\omega : t_n\le_{\mathbb B_1} b\}\]
is $\mathbb B_1/g$-generic over $V[g]$. By construction, $t\in h$ and $\sigma(t_n)\in g_3$ for all $n<\omega$. Since $\sigma$ is a projection and $t\in h$, then $\sigma''(h_{\downarrow t})$ is $\mathbb B_3$-generic. Also, $\sigma''(h_{\downarrow t})\subseteq g_3$ so, by maximality, $\sigma''(h_{\downarrow t})= g_3$. Hence, $h$ has the desired capturing feature.
\end{proof}

\section{Characterising Higher Solovay Models}

In this section we prove the main theorem of this article, the higher analogue of Theorem \ref{BagariaWoodin}.

\begin{theorem}\label{TWgen}
Suppose $V\subseteq M$ are models of \textsf{ZFC} and assume that $L(V_{\kappa+1})^M$ is a $\kappa$-Solovay model over $V$. Let $\mathbb P_\infty$ be a $(\kappa,(\kappa^+)^M)$-nice weak $\Sigma$-system in $V$. Then there is a forcing notion $\mathbb W$ in $M$ (the \emph{constellating poset}) 
such that in the forcing extension by $\mathbb W$ there is a filter $C\subseteq\mathbb P_\infty$ generic over $V$ such that $M$ and $V[C]$ have the same subsets of $\kappa$, hence the same $L(V_{\kappa+1})$. 
\end{theorem}

\begin{proof}Let us denote by $\mathbb P_d$, $\pi_{\infty,d}$ and $\pi_{e,d}$ and  the same sets and maps as in the discussion after Definition \ref{def: nice systems}, for $d$ object in a downwards directed posetal category $\mathsf D$ and $e$ such that $e\to d$. Let also $\mathbb B_\infty$, $\mathbb B_d$, $\varrho_{\infty,d}$ and $\varrho_{e,d}$ be the images of the corresponding objects and morphisms under the functor $\Psi$, where we are considering the posets $\mathbb P_d$ pointed with $p_d=\pi_{\infty,d}(\mathbb 1_{\mathbb P_\infty})$, as we did after Definition \ref{miau}.\medskip

Working in $M$, we define $\mathbb W$ as the collection of all triples $\langle p,d,g\rangle\in M$ such that
\begin{itemize}
    \item $d$ is an object of $\mathsf D$;
    \item $g\subseteq\mathbb B_d$ is a generic filter over $V$;
    \item $p\in \mathbb B_\infty/g$.
\end{itemize}
The order between the conditions is $\langle q,e,h\rangle\le \langle p,d,g\rangle$ if and only if $q\le p$, $d=e$ and $g=h$; or if $q\le p$, $e\to d$ with $e\neq d$ and $h\subseteq\mathbb B_e/g$ is a generic filter over $V[g]$. Let us start by showing that $\mathbb W\neq\varnothing$, and something more, in the following Claim \ref{52}:

\begin{claim}\label{52}Given $d$ an object of $\mathsf D$ and $p\in\mathbb B_\infty$, there is $g\in M$ with $\langle p,d,g\rangle\in \mathbb W$.\end{claim}


\begin{proof}
Apply the Interpolation Lemma \ref{lemma: interpolation} to the models $V\subseteq M$, posets $\mathbb P_1:=\mathbb P_d$ and $\mathbb P_2:=\{\varrho_{\infty,d}(p)\}$, the weak projection $\pi:\mathbb{P}_{d}\rightarrow\{\varrho_{\infty,d}(p)\}$ sending everything to $\varrho_{\infty,d}(p)$, the trivial $\mathbb B_2$-generic $\{\varrho_{\infty,d}(p)\}$ and the condition $\varrho_{\infty,d}(p)\in \mathbb B_d$.\medskip

 Our Interpolation Lemma \ref{lemma: interpolation} finds us a generic $g\subseteq\mathbb B_d/\{\varrho_d(p)\}=\mathbb B_d$ in $M$ with $\varrho_d(p)\in g$. By definition of the constellating poset, it holds $\langle p,d,g\rangle\in\mathbb W$.
\end{proof}



Now let $H$ be a $\mathbb W$-generic filter over $M$. Then we define, in $M[H]$,
\[C:=\{p\in \mathbb B_\infty: \langle q,d,g\rangle\in H\ \text{for some}\ q\le p, d, g\}.\]

\begin{claim}
    $C$ is a $\mathbb B_\infty$-generic filter over $V$.
\end{claim}

\begin{proof}First of all, $C\neq\varnothing$ because $H\neq\varnothing$ and if $\langle q,d,g\rangle\in H$ then $q\in C$. Then, if $p$ and $p'$ are in $C$, it holds $\langle q,d,g\rangle\in H$ and $\langle q',d',g'\rangle\in H$ for $q\le p$ and $q'\le p'$, so since $H$ is a filter we can find some $\langle q'', d'',g''\rangle$ in $H$ that is below both of the conditions, and $q''\le p$, $q''\le p'$ with $q''\in C$. Finally, if $p\in C$ and $p\le p'$, the same $\langle q,d,g\rangle$ witnessing $p\in C$ holds for $p'$, leading to $p'\in C$. This shows $C$ is a filter.\medskip

Now let us prove $\mathbb B_\infty$-genericity. Pick $D\subseteq\mathbb B_\infty$ dense open and in $V$, and let
\[D':=\{\langle q,e,h\rangle\in \mathbb W: q\in D\}.\]
We shall show $D'$ is dense in $\mathbb W$. Indeed, let $\langle q,e,h\rangle\in \mathbb W$. Then $q\in \mathbb B_\infty$ hence, by density of $D$ and by knowing that $q\in \mathbb B_\infty/h$, there is $q'\le q$ with $q'\in D\cap \mathbb B_\infty/h$. It follows that $\langle q',e,h\rangle\le \langle q,e,h\rangle$ and $\langle q',e,h\rangle\in D'$, hence $D'$ is dense in $\mathbb W$.\medskip

Since $H$ is $\mathbb W$-generic over $M$, we have $D'\cap H\neq\varnothing$. Now consider any element $\langle q,d,g\rangle\in D'\cap H$, then it holds that $q\in D\cap C$, hence $C$ is $\mathbb B_\infty$-generic over $V$.
%
%
%
\end{proof}

To finish the proof of Theorem \ref{TWgen}, we only need to prove that $M$ and $V[C]$ have the same subsets of $\kappa$, that is, $\wp(\kappa)^M=\wp(\kappa)^{V[C]}$. Indeed, in the statement of the theorem a filter $C$ generic for the poset $\mathbb P_\infty$ is required, but $\mathbb P_\infty$ and $\mathbb B_\infty=\mathrm{RO}(\mathbb P_\infty)$ are forcing equivalent, hence our filter $C$ that is $\mathbb B_\infty$-generic over $V$ is enough.


\begin{claim}\label{54}
    For all $x\in\wp(\kappa)^M$, the set \[D_x:=\{\langle q,d,g\rangle\in\mathbb W: x\in V[g]\}\] is dense.
\end{claim}

\begin{proof}
    Fix $x\in\wp(\kappa)^M$ and let $\langle p,d,h\rangle$ be a condition in $\mathbb W$. We want to show that there is $\langle q,e,h'\rangle\in\mathbb{W}$ with $\langle p,d,h\rangle\ge\langle q,e,h'\rangle$ and crucially $x\in V[h']$.\medskip
    
    By hypothesis, there is a $\Sigma$-Příkrý forcing notion $\mathbb Q$ in $V$, which has cardinality $|\mathbb Q|\le \kappa$ in $M$, and some $\mathbb Q$-generic filter $g\in M$ over $V$ such that $x\in V[g]$.\medskip
    
    Let $G$ be a filter $\mathbb{B}_\infty/h$-generic over $V[h]$ with $p\in G$. Notice that this $G$ need not to be in $M$. We claim there is a $\mathbb{Q}$-generic $g'\in V[h][G]=V[G]$, and to show it we are going to invoke the $\Sigma$-Příkrý nature of $\mathbb{Q}$ and its ``interpolation feature".\medskip

    More concretely, notice that $L(V_{\kappa+1})^{V[G]}$ is $\kappa$-Solovay over $V$, hence we can apply the Interpolation Lemma \ref{lemma: interpolation} to the models $V\subseteq V[G]$, the posets $\mathbb P_1:=\mathbb Q$ and $\mathbb P_2:=\{\mathbb 1_\mathbb Q\}$, the trivial weak projection $\pi:\mathbb Q\to\{\mathbb 1_\mathbb Q\}$ that sends everything to $\mathbb 1_\mathbb Q$, and the condition $\mathbb 1_\mathbb Q$. Then our Interpolation Lemma \ref{lemma: interpolation} yields a generic filter $g'_1\subseteq \mathbb R/\{\mathbb 1_\mathbb Q\}=\mathbb R$ in $V[G]$, where $\mathbb R:=\mathrm{RO}(\mathbb Q)$. Let $g':=g_1'\cap\mathbb Q$.\medskip

    Thus, $g'\in V[G]$ is $\mathbb Q$-generic over $V$. Now we exploit the $\kappa$-capturing property of $\mathcal P$, to find an object $e$ of $\mathsf D$ such that $e\to d$ and $g'\in V[h][\varrho_{\infty,e}''G]$. Hence,
    \[V[h][\varrho_{\infty,e}''G]\models ``g'\text{ is }\mathbb{Q}\text{-generic over}\ V".\]
    Now, recall that a forcing name $\tau$ is \emph{open} whenever $\langle \sigma,q\rangle\in \tau$ implies $\langle \sigma,r\rangle\in\tau$ for all $r\le q$. For every name $\sigma$ there is always an open name $\tau$ such that $\mathbb 1\Vdash (\tau=\sigma)$. Hence, fix $t\in \varrho_{\infty,e}''G\subseteq\mathbb B_e/h$ and let us pick an open $\mathbb B_e/h$-name $\gamma$ such that
    \[V[h]\models ``t \Vdash_{\mathbb B_e/h} ``\gamma\ \text{is}\ \mathbb Q\text{-generic over}\ V"".\]
Since $\gamma$ is open, we may assume that it is actually a $(\mathbb B_e/h)_{\downarrow t}$-name, for otherwise
\[\gamma':=\{\langle \sigma, v\rangle: v\le t,\ \langle \sigma, v\rangle\in \gamma\}\]
is a $(\mathbb B_e/h)_{\downarrow t}$-name with $t\Vdash_{\mathbb B_e/h} (\gamma=\gamma')$, thanks to the openness of $\gamma$.\medskip


Since $t\in\varrho_{\infty,e}''G$ and $p\in H$, we may also assume $t\le \varrho_{\infty,e}(p)$. For otherwise, say $\varrho_{\infty,e}(s)\le t$ for some $s\in H$. Then we can find some $s'\le s$ and $s'\le p$ with $s'\in H$. Replace $t$ with $t':=\varrho_{\infty,e}(s')\le\varrho_{\infty,e}(p)$, and we get the desired conclusion.\medskip

Working for a moment inside $V[h]$, let us consider the complete embedding \[\imath:\mathbb{R}\rightarrow(\mathbb{B}_{e}/h)_{\downarrow t}\] induced by $\gamma$, namely, the map defined by the formula
$$\imath(r):=\bigvee\{s\in(\mathbb{B}_{e}/h)_{\downarrow t}: s\Vdash_{(\mathbb{B}_{e}/h)_{\downarrow t}}\check{r}\in\gamma\}\big.$$ Appealing to Lemma \ref{karagila} we deduce (inside $V[h]$) the existence of a projection $$\sigma\colon(\mathbb{B}_{e}/h)_{\downarrow t}\rightarrow\mathbb R.$$ 
By the Capturing Lemma \ref{nocap}, applied to $V\subseteq M$, the projection $\sigma$ and the generic $g\in M$, we deduce there is a generic $h'\subseteq\mathbb B_e/h$ in $M$ with $t\in h'$ and $\sigma''(h'_{\downarrow t})=g$. Since $\sigma\in V[h]$ and $h'\in V[h][h']$, we get that $g\in V[h][h']=V[h']$. So we conclude that $\langle p,e,h'\rangle\le \langle p,d,h\rangle$ by definition, and $x\in V[g]\subseteq V[h']$, hence $D_x$ is dense.
\end{proof}

Now we are able to prove that $\wp(\kappa)^M\subseteq \wp(\kappa)^{V[C]}$. Let $x\in\wp(\kappa)^M$. Since $D_x$ is dense and $H$ is generic, $D_x\cap H\neq\varnothing$, hence we can pick $\langle q,d,g\rangle\in D_x\cap H$.\medskip

Fix $p\in C$. Then there is $\langle p',e,h\rangle\in H$ where $p'\le p$, $e$ is an object of $\mathsf D$ and $g$ is $\mathbb B_e$-generic. Since $H$ is a filter, there is a common extension $\langle p'',e',h'\rangle$ of $\langle p',e,h\rangle$ and $\langle q,d,g\rangle$. We know $p''\in \mathbb B_\infty/h'$, hence $\varrho_{\infty,e'}(p'')\in h'\subseteq \mathbb B_{e'}/g$. But now, we know
\[\varrho_{\infty,d}(p'')=\varrho_{e',d}(\varrho_{\infty,e'}(p''))\in g,\]
and from $\varrho_{\infty,d}(p'')\in g$ we get that $\varrho_{\infty,d}(p)\in g$. But $p$ was picked as any element of $C$, hence we proved $\varrho_{\infty,d}''C\subseteq g$ and since they are both generic $\varrho_{\infty,d}''C= g$.\medskip

But now $x\in V[g]=V[\varrho_{\infty,d}''C]\subseteq V[C]$, hence $\wp(\kappa)^M\subseteq \wp(\kappa)^{V[C]}$, as desired.



\begin{claim}
For all objects $e$ of $\mathsf D$, the set \[E_e:=\{\langle p,e',h\rangle\in\mathbb{W}: e'\to e \}\] is dense.
\end{claim}

\begin{proof}
Let $\langle q_0,e_0,h_0\rangle\in\mathbb W$. Since the posetal category $\mathsf D$ is downwards directed, there is an object $e'$ such that $e'\to e_0$ and $e'\to e$. We want to find $h$ such that $\langle q_0,e',h\rangle\le \langle q_0,e_0,h_0\rangle$. The main obstacle is to find a $\mathbb B_{e'}$-generic $h$ over $V$ which is also $\mathbb B_{e'}/h_0$ over $V[h_0]$.\medskip

Here the Interpolation Lemma \ref{lemma: interpolation} comes to our rescue. Since $\mathbb P_\infty$ is a $(\kappa,(\kappa^+)^M)$-nice weak $\Sigma$-system, the hypotheses are fulfilled when regarded with respect to $V\subseteq M$, the posets $\mathbb P_1:=\mathbb P_{e'}$ and $\mathbb P_2:=\mathbb P_{e_0}$, the weak projection $\pi_{e',e_0}$ and the condition $\varrho_{\infty,e'}(q_0)\in\mathbb B_{e'}/h_0$. The Interpolation Lemma \ref{lemma: interpolation} thus gives us a $\mathbb B_{e'}/h_0$-generic $h\in M$ with $\varrho_{\infty,e'}(q_0)\in h$. Hence $\langle q_0,e',h\rangle\in\mathbb W$ is below $\langle q_0,e_0,h_0\rangle$.
%
\end{proof}

To finish the proof, we just need $\wp(\kappa)^{V[C]}\subseteq \wp(\kappa)^M$. Let $x\in \wp(\kappa)^{V[C]}$. Thanks to the $\kappa$-capturing property, we know that there is an object $e$ of $\mathsf D$ such that $x\in V[\varrho_{\infty,e}''C]$. Now we know that $E_e$ is dense and $H$ is generic, so we can pick $\langle q,e',g\rangle\in H\cap E_e$.\medskip

As we did previously, we can prove that $\varrho_{\infty,e'}''C=g$, and we know $g\in M$, so
\[x\in V[\varrho_{\infty,e}''C]\subseteq V[\varrho_{\infty,e'}''C]=V[g]\subseteq M.\]

We just need to prove that if $x\in\wp(\kappa)^{V[C]}$ then $x\in M$. Thanks to the $\kappa$-capturing property of $(\kappa,(\kappa^+)^M)$-nice weak systems, we know there is an object $e$ in $\mathsf D$ such that $x\in{V[\varrho_e''C]}$. 
If now we pick any $\langle q,e',g\rangle \in H\cap E_e$, we know that $\varrho_{e'}''C=g$ as before, and $g\in M$, so $x\in V[\varrho_{e}''C]\subseteq V[\varrho_{e'}''C]=V[g]\subseteq M$, hence $x\in \wp(\kappa)^M$.
\end{proof}

Thanks to this result, we have a complete characterisation of $\kappa$-Solovay models:

\begin{corollary}\label{66}
$L(V_{\kappa+1})^M$ is a $\kappa$-Solovay model if and only if it is the $L(V_{\kappa+1}$) of a generic extension of $V$ via a $(\kappa,\lambda)$-nice weak $\Sigma$-system for some $\lambda>\kappa$ inaccessible.\end{corollary}

\begin{proof}
One implication is Theorem \ref{k-welldef}, the other one is Theorem \ref{TWgen}.
\end{proof}

While the Lévy collapse $\Coll(\omega, {<}\lambda)$ is uniquely determined by $\lambda$, there could be many different $(\kappa,\lambda)$-nice weak $\Sigma$-systems even with the same $\kappa$ and $\lambda$. Hence, a priori, this could have been a difference between classical Solovay models and our $\kappa$-Solovay models. But thanks to our main theorem, this is not the case:

\begin{corollary}\label{56}
Let $\mathbb P_\infty$ and $\mathbb Q_\infty$ be two $(\kappa,\lambda)$-nice weak $\Sigma$-systems. Then for all $G\subseteq\mathbb P_\infty$ that are $V$-generic there is $H\subseteq\mathbb Q_\infty$, also $V$-generic, such that \[L(V_{\kappa+1})^{V[G]}=L(V_{\kappa+1})^{V[H]}.\]
\end{corollary}

\begin{proof}
    Consider the constellating poset $\mathbb W$ of $\mathbb Q_\infty$ in $V[G]$. Thanks to Theorem \ref{TWgen} we know that in any extension by this forcing there is a filter $H\subseteq\mathbb Q_\infty$ that is $V$-generic such that $\wp(\kappa)^{V[G]}=\wp(\kappa)^{V[H]}$, hence also $L(V_{\kappa+1})^{V[G]}=L(V_{\kappa+1})^{V[H]}$, as desired.
\end{proof}
\section{Regularity Properties}

In the classical case, it is provable that if $L(\mathbb R)^M$ is a Solovay model over $V$, then every subset of the reals in $L(\mathbb R)^M$ is Lebesgue measurable, has the Baire and the perfect set properties, and satisfies many other nice regular behaviours.\medskip

We now establish an analogous result for the $\kappa$-Solovay model. To this end, we first introduce the higher analogues of the perfect set property and the Baire property in this new setting. In contrast, recent work by Agostini, Barrera and Dimonte (in \cite{ABD}) might show that a version ofLebesgue measurability for $\kappa$-reals might be unfeasible.\medskip

Classical Descriptive Set Theory (of which a great textbook is \cite{Kec}) studies the properties of definable subsets of Polish spaces, with the real line \(\mathbb{R}\), the Baire space $\omega^\omega$ and the Cantor space $2^\omega$ serving as paradigmatic examples. A number of results in the field show that simply definable sets, such as Borel or analytic sets, exhibit a rich and well-structured canonical theory.\medskip

The emerging field of Generalised Descriptive Set Theory, although the term ``generalised'' is somewhat misleading and might be better replaced with ``higher'', arises from the study of definable objects beyond the continuum. Specifically, this line of research employs descriptive set-theoretic tools to study definable sets in higher function spaces like $2^\kappa$ and $\kappa^\kappa$, as in \cite{Friedman}. It turns out that when $\kappa$ is an infinite strong limit cardinal with $\cof(\kappa)=\omega$, there are a number of results that parallel the findings of classical descriptive set theory, and that get lost in the uncountable regular case.\medskip

For the rest of the section, we assume that $\kappa$ is a cardinal with $\cof(\kappa)=\omega$ and $\beth_\kappa=\kappa$ (recall that $\beth_\kappa=\left|V_\kappa\right|$). Indeed, this implies that
\[L(V_{\kappa+1})=L(\wp(\kappa)),\]
and also if $M$ is a model of \textsf{ZFC} it holds that
\[L(V_{\kappa+1})^M\models \mathsf{ZF}+\mathsf{DC}_\kappa.\]
For a proof of $\mathsf{DC}_\kappa$ in $L(V_{\kappa+1})$ see \cite[Lemma 4.10]{DimonteRankIntoRank}.\medskip

\begin{definition}\label{polakko} A topological space is called \emph{$\kappa$-Polish} in case it is homeomorphic to a completely metrisable space with weight $\kappa$.\end{definition}

For more details on the definition of $\kappa$-Polish spaces and on the following examples, we refer to \cite{DimonteMotto}.\medskip

The first example of $\kappa$-Polish space is the higher Baire space $\kappa^\omega$, equipped with the bounded topology: the basic open neighbourhoods are the ones of the form
\[\mathcal N_s:=\{x\in\kappa^\omega: (x\upharpoonright \ell(s))=s\},\]
for $s\in \kappa^{<\omega}$. Another example, that is homeomorphic to $\kappa^\omega$, is the higher Cantor space $2^\kappa$ with the bounded topology.\medskip

A less straightforward example of $\kappa$-Polish space is the following. Let $\Sigma=\langle \kappa_n:n<\omega\rangle$ be an increasing sequence of regular cardinals with $\kappa=\sup_{n}\kappa_n$. The space
$\textstyle \prod_{n<\omega}\kappa_n$ is a closed subspace of $\kappa^\omega$, hence it is $\kappa$-Polish.\medskip


Next, we want to introduce the higher versions of the perfect set and the Baire properties. Recall that a subset of $\mathbb R$ has the perfect set property if it is either countable or contains a perfect subset, that is, a closed set without isolated points.

\begin{definition}
A subset $P$ of a $\kappa$-Polish space is \emph{$\kappa$-perfect} if it is closed and for all $x\in P$ and all open neighbourhoods $U$ of $x$ it holds $|P\cap U|\ge \kappa$. A subset $A$ of a $\kappa$-Polish space has the \emph{$\kappa$-perfect set property} if either $|A|\le \kappa$ or $A$ contains a $\kappa$-perfect subset.
\end{definition}


Since continuous maps are uniquely determined by their behaviour on any dense set, and every $\kappa$-Polish space has weight $\kappa$,  all continuous maps between $\kappa$-Polish spaces can be coded by elements of $V_{\kappa+1}$ or of $\wp(\kappa)$. In particular, this means that the $\kappa$-perfect set property is absolute between models that share the same $\wp(\kappa)$. Specifically, a set $A\in L(V_{\kappa+1})$ which is a subset of a $\kappa$-Polish space has the $\kappa$-perfect set property in $V$ if and only if it has the $\kappa$-perfect set property in $L(V_{\kappa+1})$.\medskip

In \cite{DPT}, the authors establish the following result:

\begin{theorem}\label{psp}
 Assume $\kappa$ is a $({<}\lambda)$-supercompact cardinal where $\lambda>\kappa$ is inaccessible. Then, in any extension by the Merimovich forcing, every subset of $\kappa^\omega$ in $L(V_{\kappa+1})$ has the $\kappa$-perfect set property.
\end{theorem}

\begin{proof}
    The proof of this theorem can be found in \cite[Corollary 4.12]{DPT}.
\end{proof}

This theorem enables us to prove the following corollary.

\begin{corollary}\label{pspsp}
    Let $V\subseteq M$ be models of \textsf{ZFC} such that $L(V_{\kappa+1})^M$ is $\kappa$-Solovay over $V$, and assume that $\kappa$ is $({<}(\kappa^+)^M)$-supercompact in $V$. Then every subset of $\kappa^\omega$ that belongs to $L(V_{\kappa+1})^M$ has the $\kappa$-perfect set property, both in $M$ and in $L(V_{\kappa+1})^M$.
\end{corollary}

\begin{proof}
We know that the Merimovich forcing built from $\kappa$ and $(\kappa^+)^M$ is a $(\kappa,(\kappa^+)^M)$-nice weak $\Sigma$-system in $V$. Hence by Theorem \ref{TWgen} we find $L(V_{\kappa+1})^M=L(V_{\kappa+1})^{V[G]}$ where $G$ is $V$-generic for the Merimovich forcing, and we apply the previous Theorem \ref{psp} to find the desired conclusion.
\end{proof}

If we denote by $\kappa\text{-}\mathsf{PSP}$ the statement ``every subset of $\kappa^\omega$ has the $\kappa$-perfect set property'', as a consequence of Corollary \ref{pspsp} we get that

\begin{corollary}
    Let $V\subseteq M$ be models of \textsf{ZFC} and let $\kappa$ be a $({<}(\kappa^+)^M)$-supercompact cardinal in $V$ with $\beth_\kappa=\kappa$. Then if $L(V_{\kappa+1})^M$ is a $\kappa$-Solovay model over $V$, it holds
    \[L(V_{\kappa+1})^M\models \mathsf{ZF}+\mathsf{DC}_\kappa+\kappa\text{-}\mathsf{PSP}.\]
\end{corollary}

Next we turn to the higher version of the Baire property. Recall that a subset of a topological space is said to be \emph{rare} if its complement contains an open dense set; it is \emph{meager} if it is a countable union of rare sets; and it has the \emph{Baire property} if it can be written as the symmetric difference of an open and a meager set.\medskip

In a similar fashion, one can define a \emph{$\kappa$-meager} set as an union of $\kappa$-many rare sets; and say that a set has the \emph{$\kappa$-Baire property} if it can be written as the symmetric difference of an open and a $\kappa$-meager set. A problem arises, because in the classical case $\mathbb R$ is not meager, while all three examples we considered after Definition \ref{polakko} are $\kappa$-meager, and this is trivially useless for our purposes.\medskip

To bypass this issue, another route is outlined in \cite{DimonteMottoShi}. Namely, the authors consider the space $\textstyle \prod_{n<\omega}\kappa_n$ with a topology different than the one mentioned previously. This will be the \emph{$\mathcal U$-Ellentuck-Příkrý} topology, but in order to define it we need particular assumptions: indeed, we require each $\kappa_n$ to be a measurable cardinal and we let $\mathcal U:=\langle \mathcal U_n:n<\omega\rangle$ be a sequence of normal non-principal measures, where each $\mathcal U_n$ is a measure over $\kappa_n$.\medskip

The easiest way to define this topology is to introduce the diagonal version of Příkrý forcing. This is a kind of $\Sigma$-Příkrý forcing, once again demonstrating the strong connection between this family of forcing notions and higher descriptive set theory.

\begin{definition}
    A condition in the diagonal Příkrý forcing $\mathbb P_\mathcal U$ is a sequence
    \[p = \langle \alpha^p_0,\dots, \alpha^p_{\ell(p)-1}, A^p_{\ell(p)}, A^p_{\ell(p)+1},\dots \rangle\]
    where
    \begin{itemize}
        \item $s^p=\langle \alpha^p_0,\dots, \alpha^p_{\ell(p)-1}\rangle\in \prod_{n<\ell(p)}\kappa_n$ is strictly increasing;
        \item $A^p_n$ belongs to $\mathcal{U}_n$ for $n\ge \ell(p)$;
        \item every $\beta\in A_{n+1}^p$ is bigger than $\kappa_n$. 
\end{itemize}
Given two conditions $p$ and $q$ in $\mathbb P_\mathcal U$, we write $p\leq q$ if
\begin{itemize}
    \item $s^p\restriction\ell(q) = s^q$;
    \item $s^p(n)\in A^q_n$ for all $\ell(q)\le n<\ell(p)$;
    \item $A^p_n\subseteq A^q_n$ for all $n\geq \ell(p)$.
    \end{itemize}
\end{definition}

Note that, when considering the diagonal Merimovich forcing from Definition \ref{meridiag}, this diagonal Příkrý forcing is equal to $\mathbb P_d$, where $d$ the terminal object in the category $\mathsf D$, namely the sequence $d=\langle \{\kappa_n\}: n<\omega\rangle$.

\begin{definition}
    For each condition $p\in\mathbb P_\mathcal U$, define
    \[\mathcal N_p:=\{x\in 
\textstyle \prod_{n<\omega}\kappa_n: (x\upharpoonright \ell(p))=s^p \land \forall n\ge \ell(p), x(n)\in A^p_n\}.\]
The \emph{$\mathcal U$-Ellentuck-Příkrý topology} is the topology on $\textstyle \prod_{n<\omega}\kappa_n$ having $\mathcal N_p$ for $p\in\mathbb P_\mathcal U$ as the basic open sets. A subset of $\textstyle \prod_{n<\omega}\kappa_n$ has the \emph{$\mathcal U$-Baire property} if it has the $\kappa$-Baire property in the $\mathcal U$-Ellentuck-Příkrý topology.
\end{definition}

As for the $\kappa$-perfect set property, having the $\mathcal U$-Baire property for a subset of $\textstyle \prod_{n<\omega}\kappa_n$ is absolute between models with the same $\wp(\kappa)$, and the proof is routine. We would like to thank Vincenzo Dimonte for bringing this fact to our attention.\medskip

Dimonte, Poveda and the third author in \cite{DPT} show the following result:

\begin{theorem}\label{Ubai}
Let $\kappa<\lambda$ be such that $\kappa=\sup_n \kappa_n$ is the supremum of a sequence of cardinals such that each $\kappa_n$ is $({<}\lambda)$-supercompact, and $\lambda$ is inaccessible. Then, in any extension by the diagonal Merimovich forcing, every subset of $\kappa^\omega$ in $L(V_{\kappa+1})$ has the $\kappa$-perfect set property, and also every subset of $\textstyle \prod_{n<\omega}\kappa_n$ in $L(V_{\kappa+1})$ has the $\mathcal U$-Baire property.
\end{theorem}

\begin{proof}
    The proof of this theorem can be found in \cite[Corollary 4.36]{DPT}.
\end{proof}

As before, we now manage to prove the following corollary.

\begin{corollary}\label{pspspsp}
    Let $V\subseteq M$ be models of \textsf{ZFC} such that $L(V_{\kappa+1})^M$ is $\kappa$-Solovay over $V$, and assume that $\kappa=\sup_n \kappa_n$ where each $\kappa_n$ is $({<}(\kappa^+)^M)$-supercompact in $V$. Then every subset of $\kappa^\omega$ that belongs to $L(V_{\kappa+1})^M$ has the $\kappa$-perfect set property, both in $M$ and in $L(V_{\kappa+1})^M$; and every subset of $\textstyle \prod_{n<\omega}\kappa_n$ that belongs to $L(V_{\kappa+1})^M$ has the $\mathcal U$-Baire property, both in $M$ and in $L(V_{\kappa+1})^M$, for $\mathcal U$ a sequence in $V$ of normal measures on the cardinals $\kappa_n$ for $n<\omega$.
\end{corollary}

\begin{proof}
We know that the diagonal Merimovich forcing built from $\kappa$ and $(\kappa^+)^M$ is a $(\kappa,(\kappa^+)^M)$-nice weak $\Sigma$-system in $V$. Hence by Theorem \ref{TWgen} $L(V_{\kappa+1})^M=L(V_{\kappa+1})^{V[G]}$ where $G$ is $V$-generic for the diagonal Merimovich forcing, and we apply Theorem \ref{Ubai} to find the desired conclusion.
\end{proof}

If we denote by $\mathcal U\text{-}\mathsf{BP}$ the statement ``every subset of $\textstyle \prod_{n<\omega}\kappa_n$ has the $\mathcal U$-Baire property'', as a consequence of Corollary \ref{pspspsp} we get that

\begin{corollary}
    Let $V\subseteq M$ be models of \textsf{ZFC} and let $\kappa=\sup_n \kappa_n$ where each $\kappa_n$ is a $({<}(\kappa^+)^M)$-supercompact cardinal in $V$, and $\beth_\kappa=\kappa$. Then, if $L(V_{\kappa+1})^M$ is a $\kappa$-Solovay model over $V$,
    \[L(V_{\kappa+1})^M\models \mathsf{ZF}+\mathsf{DC}_\kappa+\kappa\text{-}\mathsf{PSP}+\mathcal U\text{-}\mathsf{BP}.\]
\end{corollary}

    \section{An Application for Absoluteness}

Here we prove our Theorem \ref{bb2}, the higher analogue of Bagaria-Bosch's Theorem \ref{bb}. This shows the existence of an elementary embedding between $\kappa$-Solovay models.

\begin{theorem}\label{bb2}
Suppose $V\subseteq M$ and $V\subseteq N$ are models of \textsf{ZFC} such that $L(V_{\kappa+1})^M$ and $L(V_{\kappa+1})^N$ are $\kappa$-Solovay models over $V$, $\wp(\kappa)^M\subseteq \wp(\kappa)^N$ and $(\kappa^+)^M=(\kappa^+)^N$. Suppose there is a $(\kappa,(\kappa^+)^M)$-nice weak $\Sigma$-system in $V$. Then there is a unique
\[j:L(V_{\kappa+1})^M\hookrightarrow L(V_{\kappa+1})^N\]
that is an elementary embedding and fixes all the ordinals.
\end{theorem}

\begin{proof}
First of all, notice that if $j$ is an embedding as in the statement, and $a\subseteq\kappa$ is in $M$, $j(a)=a$. Indeed, since $\wp(\kappa)^M\subseteq \wp(\kappa)^N$, the element $j(a)$ is also a subset of $\kappa$ in $N$, and $\alpha\in j(a)$ if and only if $\alpha\in a$ for all ordinals $\alpha$, since $j(\alpha)=\alpha$.\medskip

Fix now some $A\in L(V_{\kappa+1})^M$. By definition of $L(V_{\kappa+1})^M$, there are an ordinal $\alpha$, a set $a\in \wp^M(\kappa)$ and a formula $\varphi$ of the language of set theory such that
\[A = \{x\in M: L(V_{\kappa+1})^M\models \varphi(x,\alpha,a)\}.\]
If the embedding $j$ as in the statement exists, then necessarily it holds that
\[j(A) = \{x\in N: L(V_{\kappa+1})^N \models \varphi(x,\alpha,a)\}.\]
So we can take this as the definition of $j$, and then we show this is well-defined and an elementary embedding, and we automatically get the uniqueness of such a $j$.\medskip

In order to do so, we only need to show that for every ordinal $\alpha$, every $a\in\wp^M(\kappa)$ and every set-theoretic formula $\varphi$ it holds the equivalence
\[L(V_{\kappa+1})^M \models \varphi(\alpha,a)\ \text{if and only}\ L(V_{\kappa+1})^N \models \varphi(\alpha,a).\hspace{1cm}(\heartsuit)\]
Indeed, assume we proved the equivalence $(\heartsuit)$. Let $A_1$ and $A_2$ be defined by
\[A_i := \{x\in M: L(V_{\kappa+1})^M\models \varphi_i(x,\alpha_i,a_i)\}\]
where $\alpha_1$ and $\alpha_2$ are ordinals, $a_1$ and $a_2$ are elements of $\wp^M(\kappa)$, and $\varphi_1$ and $\varphi_2$ are formul\ae\ in the language of set theory. Then it holds that $A_1=A_2$ if and only if
\[L(V_{\kappa+1})^M \models \forall x\ (\varphi_1(x,\alpha_1,a_1) \leftrightarrow \varphi_2(x,\alpha_2,a_2))\]
and applying the equivalence $(\heartsuit)$ this holds if and only if
\[L(V_{\kappa+1})^N \models \forall x\ (\varphi_1(x,\alpha_1,a_1) \leftrightarrow \varphi_2(x,\alpha_2,a_2)),\]
since we can code $\alpha_1$ and $\alpha_2$ into an ordinal $\alpha$, and $a_1$ and $a_2$ into some $a\in\wp^M(\kappa)$. By definition of $j$, this is equivalent to $j(A_1)=j(A_2)$, proving that $j$ is well-defined. Following the same proof strategy, we get that $A_1\subseteq A_2$ if and only if
\[L(V_{\kappa+1})^M \models \forall x\ (\varphi_1(x,\alpha_1,a_1) \to \varphi_2(x,\alpha_2,a_2))\]
and thanks to $(\heartsuit)$ this is equivalent to
\[L(V_{\kappa+1})^N \models \forall x\ (\varphi_1(x,\alpha_1,a_1) \to \varphi_2(x,\alpha_2,a_2)),\]
which in turn is equivalent to $j(A_1)\subseteq j(A_2)$ by definition. Since the relation $\in$ is definable from the relation $\subseteq$, we get that $j$ is an elementary embedding.\medskip

Hence, we just need to prove the equivalence $(\heartsuit)$. Let $\mathbb P_\infty$ be a $(\kappa,(\kappa^+)^M)$-nice weak $\Sigma$-system. Since $L(V_{\kappa+1})^M$ is $\kappa$-Solovay over $V$, we can consider its constellating poset $\mathbb W^M\in M$ of $\mathbb P_\infty$ in $M$ as in the proof of Theorem \ref{TWgen}.\medskip

Since $(\kappa^+)^N=(\kappa^+)^M$, $\mathbb P_\infty$ is also a $(\kappa, (\kappa^+)^N)$-nice weak $\Sigma$-system in $V$, and $L(V_{\kappa+1})^N$ is a $\kappa$-Solovay model over $V$, so we can consider the constellating poset of $\mathbb P_\infty$ over $N$ as in the proof of Theorem \ref{TWgen}, and call it $\mathbb W^N\in N$.\medskip

Recall that we are trying to prove $(\heartsuit)$, so let us fix an ordinal $\alpha$, a subset $a\subseteq\kappa$ in $M$ and a formula $\varphi$ of set theory. Thanks to the Claim \ref{54} of the proof of Theorem \ref{TWgen}, we know there are an object $d$ in $\mathsf D$ and $g\subseteq\mathbb B_d$ that is $V$-generic with $a\in V[g]$, $g\in M$.\medskip

Since the transitive closure of $\mathbb B_d$ has size $\kappa$ in $M$ thanks to the fact $\mathbb P_\infty$ is $(\kappa^+)^M$-bounded, we can code $g$ into a $\kappa$-real of $M$, hence the code is also in $N$, so also $g\in N$. Notice that for all $p\in\mathbb B_\infty/g$ it holds $\langle p,d,g\rangle\in \mathbb W^M\cap \mathbb W^N$.\medskip

Assume $L(V_{\kappa+1})^M\models \varphi(\alpha,a)$, or equivalently $M\models\varphi(\alpha,a)^{L(V_{\kappa+1})}$. Let $H_M\subseteq\mathbb W^M$ be $M$-generic, $C_M\in M[H_M]$ be a $\mathbb P_\infty/g$-generic filter over $V$, as defined by Theorem \ref{TWgen}. Then it holds that $L(V_{\kappa+1})^M=L(V_{\kappa+1})^{V[g][C_M]}$ and so 
\[V[g][C_M]\models \varphi(\alpha,a)^{L(V_{\kappa+1})}.\]
In particular there is some $p\in C_M$ such that \[V[g]\models (p \Vdash_{\mathbb P_\infty/g} \varphi(\check \alpha, \check a)^{L(V[\dot{G}]_{\kappa+1})}).\]
Now, since $\langle p,d,g\rangle \in \mathbb W^M\cap \mathbb W^N$, let us pick a $N$-generic filter $H_N\subseteq\mathbb W^N$ with $\langle p,d,g\rangle\in H_N$. And let $C_N\in N[H_N]$ be a $\mathbb P_\infty/g$-generic filter over $V$ with $p\in C_N$, defined as in the proof of Theorem \ref{TWgen}, so that $L(V_{\kappa+1})^N=L(V_{\kappa+1})^{V[g][C_N]}$. Now we have a $p\in C_N$ such that 
\[V[g]\models (p \Vdash_{\mathbb P_\infty/g} \varphi(\check \alpha, \check a)^{L(V[\dot{G}]_{\kappa+1})}),\]
hence by the forcing theorem it holds that $V[g][C_N]\models \varphi(\alpha,a)^{L(V_{\kappa+1})}$, or equivalently \[L(V_{\kappa+1})^{V[g][C_N]}\models\varphi(\alpha,a).\] Therefore, thanks to Theorem \ref{TWgen} we proved that $L(V_{\kappa+1})^{N}\models\varphi(\alpha,a)$.\medskip

Now assume $L(V_{\kappa+1})^N\models \varphi(\alpha,a)$. Let $H_N\subseteq\mathbb W^N$ be a $N$-generic filter and let $C_N\in N[H_N]$ be a $\mathbb P_\infty/g$-generic filter over $V$, defined by Theorem \ref{TWgen}. Thus \[V[g][C_N]\models \varphi(\alpha,a)^{L(V_{\kappa+1})},\] so in particular there is some $p\in C_N$ such that \[V[g]\models (p \Vdash_{\mathbb P_\infty/g} \varphi(\check \alpha, \check a)^{L(V[\dot{G}]_{\kappa+1})}).\]
Now since $\langle p,d,g\rangle \in \mathbb W^M\cap \mathbb W^N$, let us pick a $M$-generic filter $H_M\subseteq\mathbb W^M$ such that $\langle p,d,g\rangle\in H_M$. And let $C_M\in M[H_M]$ be a $\mathbb P_\infty/g$-generic filter over $V$ with $p\in C_M$, defined as in the proof of Theorem \ref{TWgen}. Now there is $p\in C_M$ with
\[V[g]\models (p \Vdash_{\mathbb P_\infty/g} \varphi(\check \alpha, \check a)^{L(V[\dot{G}]_{\kappa+1})}),\]
hence by the forcing theorem we get that $L(V_{\kappa+1})^{V[g][C_M]}\models \varphi(\alpha,a)$ and so by our Theorem \ref{TWgen} we have $L(V_{\kappa+1})^{M}\models\varphi(\alpha,a)$, as desired, concluding the proof.
%
%
%
%
%
%
%
\end{proof}

As a consequence, we get the following corollary:

\begin{corollary}\label{82}
Suppose $V\subseteq M$ and $V\subseteq N$ are models of \textsf{ZFC} such that $L(V_{\kappa+1})^M$ and $L(V_{\kappa+1})^N$ are $\kappa$-Solovay models over $V$, $\wp(\kappa)^M\subseteq \wp(\kappa)^N$ and $(\kappa^+)^M=(\kappa^+)^N$. Assume $\kappa$ is $({<}(\kappa^+)^M)$-supercompact in $V$. Then there is a unique
\[j:L(V_{\kappa+1})^M\hookrightarrow L(V_{\kappa+1})^N\]
that is an elementary embedding and fixes all the ordinals.
\end{corollary}

\begin{proof}
    We can apply Lemma \ref{merinonfully} to deduce that under these hypotheses there is a $(\kappa,(\kappa^+)^M)$-nice weak $\Sigma$-system in $V$, and then apply Theorem \ref{bb2}.
\end{proof}

\section{Open Questions}

In Corollary \ref{82}, we are assuming that $\kappa$ is $({<}\lambda)$-supercompact for some inaccessible $\lambda$ because this is the least large cardinal hypothesis under which we know for sure there is a $\kappa$-Solovay model, thanks to Lemma \ref{merinonfully} and Theorem \ref{TWgen}. We ask whether this is optimal.

\begin{question}\label{1}
    What is, modulo \textsf{ZFC}, the exact consistency strength of the existence of a $\kappa$-Solovay model over some inner model $V$?
\end{question}

Note that, in light of Corollary \ref{66}, the consistency strength of the existence of a $\kappa$-Solovay model is the same as the one of the existence of a $(\kappa,\lambda)$-nice weak $\Sigma$-system. Therefore, Question \ref{1} can be rephrased as a first-order question in the language of \textsf{ZFC}.\medskip

Barrera, Dimonte and Müller in \cite{Nando} showed that, if the $\kappa$-perfect set property holds at $\kappa$-coanalytic sets, then there is an inner model with $\kappa$-many measurable cardinals. This gives us a lower bound on the large cardinal hypothesis for Question \ref{1}.\medskip

A remarkable consequence of $I_0(\kappa)$ discovered by Woodin is that $\kappa^+$ is measurable in $L(V_{\kappa+1})$ as witnessed by restrictions of the club filter on $\kappa^+$. This mirrors the fact that $\omega_1$ is measurable in $L(\mathbb{R})$ under $\textsf{AD}^{L(\mathbb{R})}$.
\begin{question}\label{question: measurability}
    Let $L(V_{\kappa+1})^M$ be $\kappa$-Solovay over $V$. Is $(\kappa^+)^{M}$ measurable in $L(V_{\kappa+1})^M$?
\end{question}

Under $I_0(\kappa)$, other easier combinatorial consequences about $L(V_{\kappa+1})$ can be proven. We ask whether the same hold in $\kappa$-Solovay models. 

\begin{question}
   Let us work in $L(V_{\kappa+1})^M$, a $\kappa$-Solovay model over $V$. Is there a scale at $\kappa$? Does $\square_\kappa$ fail? Is it true that there is no $\kappa^+$-Aronszajn tree? Is it true that there is no $\kappa^+$-sequence of distinct elements of $V_{\kappa+1}$? 
\end{question}

Another important result about classical Solovay models is their correctness under large cardinal hypotheses, meaning that under some of these hypotheses the inner model $L(\mathbb R)$ of $V$ is elementarily equivalent to any Solovay model over $V$. This fact is a consequence of remarkable work by Saharon Shelah and Hugh Woodin from \cite{SW}.

\begin{question}
    Under large cardinal hypothesis, is it true that the $L(V_{\kappa+1})$ of $V$ is elementary equivalent to any $\kappa$-Solovay model over $V$?
\end{question}

The analysis conducted in Section 5 opens a potentially fruitful line of inquiry: the investigation of regularity properties and combinatorial principles at small singular cardinals, with $\aleph_\omega$	functioning as a prototypical case. To date, it is unclear whether $\aleph_\omega$-Solovay models exist. This problem underscores a broader methodological challenge, namely, the nontriviality of downward transfer principles in set theory, particularly when attempting to reflect combinatorial configurations established at higher singular cardinals $\kappa$ to the lower one (that is $\aleph_\omega$). We are currently unable to determine how to show something similar. A vague question that summarises this line of inquiry is the following: 

\begin{question}\label{question: aleph_omega}
    Modulo large cardinals, is there a forcing extension $M$ of $V$ such that $L(\wp(\aleph_\omega))$ is some kind of higher Solovay model over $V$?
\end{question}

With a view to Questions \ref{question: measurability} and \ref{question: aleph_omega} above, it is natural to look at the measurability of $\aleph_\omega^+$.

\begin{question}
Can it be consistent, modulo large cardinals, a configuration where $\aleph_\omega$ is a strong limit cardinal and $\aleph_\omega^+$ is measurable in $L(\wp({\aleph_\omega}))$?
\end{question}

In \cite{AP25}, Poveda and the second author produced a model of \textsf{ZFC} where $\aleph_\omega$ is a strong limit cardinal and the inner model $L(\wp({\aleph_\omega}))$ satisfies many interesting Solovay-like properties: every subset $A\subseteq{}^\omega\aleph_\omega$ has the $\aleph_\omega$-perfect set property; there is no scale at $\aleph_\omega$; the Singular Cardinal Hypothesis fails at $\aleph_\omega$; and  Shelah's Approachability property fails at $\aleph_\omega$.\medskip

However, due to the lack of $\aleph_\omega$-analogues of $I_0(\kappa)$, the analogous configuration at the level of the first singular cardinal is not even known to be consistent.

\section*{Acknowledgements}

The first author is a member of the Gruppo Nazionale per le Strutture Algebriche, Geometriche e le loro Applicazioni (GNSAGA) of the Istituto Nazionale di Alta Matematica (INdAM). The second author was supported by the Italian PRIN 2022 Grant ``Models, sets and classifications'' and by the European Union -- Next Generation EU.\medskip

We would like to wholeheartedly thank Maria Bevilacqua for her many insightful comments on the category-theoretic statements and definitions, which greatly helped improve the readability and clarity of the paper.

\end{document}